\documentclass{article}

\usepackage{latexsym}
\usepackage{amssymb}
\usepackage{amsthm}
\usepackage{amscd}
\usepackage{amsmath}
\usepackage{mathrsfs}
\usepackage{graphicx}
\usepackage{hyperref}
\usepackage{shuffle}

\usepackage[all]{xy}
\input xy \xyoption{frame}

\usepackage[colorinlistoftodos]{todonotes}
\usetikzlibrary{chains,scopes,decorations.markings}

\theoremstyle{definition}
\newtheorem* {theorem*}{Theorem}
\newtheorem* {conjecture*}{Conjecture}
\newtheorem{theorem}{Theorem}[section]

\theoremstyle{definition}

\theoremstyle{definition}

\newtheorem* {example*}{Example}

\newtheorem{lemma}[theorem]{Lemma}
\theoremstyle{definition}
\newtheorem{definition}[theorem]{Definition}
\theoremstyle{definition}

\newtheorem{conjecture}[theorem]{Conjecture}
\newtheorem{proposition}[theorem]{Proposition}
\newtheorem{corollary}[theorem]{Corollary}

\theoremstyle{definition}
\newtheorem {example}[theorem]{Example}
\theoremstyle{definition}

\theoremstyle{definition}

\theoremstyle{definition}
\newtheorem{problem}[theorem]{Problem}
\theoremstyle{definition}

\def\modu{\ (\mathrm{mod}\ }

\def\({\left(}
\def\){\right)}

\newcommand{\QQ}{\mathbb{Q}}

\def\NN{\mathbb{N}}

\def\ZZ{\mathbb{Z}}

\def\spanning{\textnormal{-span}}

\newcommand{\cM}{\mathcal{M}}
\newcommand{\cN}{\mathcal{N}}

\newcommand{\sgn}{\mathrm{sgn}}

\def\fk{\mathfrak}

\def\barr{\begin{array}}
\def\earr{\end{array}}
\def\ba{\begin{aligned}}
\def\ea{\end{aligned}}
\def\be{\begin{equation}}
\def\ee{\end{equation}}

\def\Cyc{\mathrm{Cyc}}

\def\qquand{\qquad\text{and}\qquad}
\def\quand{\quad\text{and}\quad}

\newcommand{\Sym}{\operatorname{Sym}}

\def\inv{\mathrm{Inv}}

\def\cH{\mathcal H}
\def\cM{\mathcal M}
\def\DesR{\mathrm{Des}_R}
\def\DesL{\mathrm{Des}_L}

\def\PP{\mathbb{P}}

\def\fkS{\fk S}

\def\ben{\begin{enumerate}}
\def\een{\end{enumerate}}

\def\cF{\mathcal F}

\def\fpf{{\mathsf {FPF}}}

\def\DesF{\mathrm{Des}_\fpf}

\def\cfpf{\hat c_{\fpf}}

\def\ellhat{\hat\ell}
\def\ellfpf{\ellhat_\fpf}
\def\Ffpf{\hat F^\fpf}

\def\x{\textbf{x}}

\newcommand{\xRightarrow}[2][]{\ext@arrow 0359\Rightarrowfill@{#1}{#2}}

\newcommand{\cA}{\mathcal{A}}

\def\cAfpf{\cA^\fpf}

\def\ellhat{\hat\ell}
\def\Ffpf{\hat F^\fpf}

\def\tS{\tilde S}

\def\tI{\tilde I}

\def\Inv{\operatorname{Inv}}

\def\iF{\hat{F}}

\def\tF{F}

\def\Sym{\textsf{Sym}}

\def\Par{\textsf{Par}}
\def\tI{\tilde I}

\def\arcstart{\ \xy<0cm,-.06cm>\xymatrix@R=.1cm@C=.2cm }
\newcommand{\arcstartc}[1]{\ \xy<0cm,-.15cm>\xymatrix@R=.1cm@C=#1cm}

\def\tp{\Theta^+}
\def\tm{\Theta^-}
\def\tpm{\{\Theta^\pm\}}

\def\cF{\tilde I^{\fpf}}
\def\cFp{\tilde I^{\fpf}_{n,+}}
\def\cFm{\tilde I^{\fpf}_{n,-}}

\def\h {\mathrm{ht}}
\def\Cyc{\mathrm{Cyc}}
\def\wCyc{\widetilde\Cyc}
\def\sfpf{\sgn_\fpf}
\def\m{\mathfrak{m}}
\def\n{\mathfrak{n}}
\usepackage{fullpage}

\numberwithin{equation}{section}
\allowdisplaybreaks[1]

\makeatletter
\renewcommand{\@makefnmark}{\mbox{\textsuperscript{}}}
\makeatother
\title{Quasiparabolic sets and Stanley symmetric functions for affine fixed-point-free involutions}
\author{
Yifeng ZHANG\thanks{
Email: \tt yzhangci@connect.ust.hk
}
\\
Department of Mathematics \\ Hong Kong University of Science and Technology
}
\date{ }
\begin{document}

\maketitle

\setcounter{tocdepth}{2}

\begin{abstract}
We introduce and study affine analogues of the fixed-point-free (FPF) involution Stanley symmetric functions of Hamaker, Marberg, and Pawlowski.
Our methods use the theory of quasiparabolic sets introduced by Rains and Vazirani, and
we prove that the subset of FPF-involutions is a quasiparabolic set for the affine symmetric group under conjugation. Using properties of quasiparabolic sets, we prove a transition formula for the affine FPF involution Stanley symmetric functions, analogous to Lascoux and Sch\"utzenberger's transition formula for Schubert polynomials. Our results suggest several conjectures and open problems.
\end{abstract}

\tableofcontents

\section{Introduction}

In this article, we study a family of symmetric functions $\Ffpf_z$ indexed by the fixed-point-free involutions in
the affine symmetric group. We prove that such affine permutations form a \emph{quasiparabolic set}, 
as predicted by Rains and Vazirani \cite{RV}. 
We prove a transition formula for the symmetric functions $\Ffpf_z$ using this result,
among other applications.

Stanley introduced a family of symmetric power series $F_w$ indexed by permutations, 
now called \emph{Stanley symmetric functions}, in \cite{Stan}. 
These functions are important since they are the stable limits of 
the \emph{Schubert polynomials} $\fkS_w$, 
which represent the cohomology classes of type A Schubert varieties. 
They can be computed using the transition formula for Schubert polynomials given by Lascoux and Sch\"utzenberger in \cite{Lascoux}. This formula implies that the Stanley symmetric functions are Schur positive, in other words, are $\NN$-linear combination of Schur functions $s_\lambda$.
Edelman and Greene first proved this positivity in \cite{EG}.

Since the 1980s, many generalizations of Stanley symmetric functions have been introduced. 
In 2006, Lam \cite{Lam} defined the \emph{affine Stanley symmetric functions}, 
which are indexed by permutations in the affine symmetric group $\tS_n$. 
These symmetric functions also have a geometric meaning: 
they represent cohomology classes for the affine Grassmannian \cite{Lam2008}. 
Affine Stanley symmetric functions are affine Schur positive in the sense that they expand positively in
terms of \emph{affine Schur functions} \cite{Lam2008}. Moreover, they satisfy a transition formula 
\cite{LamShim}, which generalizes the transition formula for ordinary Stanley symmetric functions.
It is expected that the affine transition formula implies affine Schur positivity, but it remains an open problem
to show that the resulting recursion terminates.

Hamaker, Marberg, and Pawlowski have considered another analogue
of Stanley symmetric functions.
In  \cite{HMP1,HMP4}, they studied the \emph{involution Stanley symmetric functions}, which are indexed by self-inverse permutations $z=z^{-1}$ in the finite symmetric group $S_n$. These functions 
are the stable limits of polynomial representatives for 
the cohomology classes of 
the orbit closures of the orthogonal group acting on the type A flag variety. Like the functions above, involution Stanley symmetric functions admit a transition formula \cite{HMP3}, and this can be used to 
show that they are Schur $P$-positive, that is, are $\NN$-linear 
combinations of \emph{Schur $P$-functions} $P_\lambda$ \cite{HMP4}.

This paper is a sequel to  \cite{MZ}, which constructs  
involution Stanley symmetric functions for affine permutations. This prior work defines
and proves a transition formula for what are called \emph{affine involution Stanley symmetric functions}.

There is another family of symmetric functions studied by Hamaker, Marberg, and Pawlowski in \cite{HMP3,HMP5}, called the \emph{fixed-point-free (FPF) involution Stanley symmetric functions}, which
are indexed by self-inverse permutations $z=z^{-1}$ with no fixed points in $S_n$. 
These symmetric functions have a similar interpretation in terms of the cohomology of orbits closures 
in the type A flag variety, but now for the action of a symplectic group. They also admit a transition formula and expand positively into Schur $P$-functions.

In this article, we now consider the affine generalization of FPF-involution Stanley symmetric functions. 
Specifically, we introduce a family of symmetric functions $\Ffpf_z$ indexed by affine fixed-point-free involutions $z\in \tS_n$. We establish several noteworthy properties of these power series,
which we expect are related to the geometry of affine analogues of certain symmetric varieties.
For example, we prove that each $\Ffpf_z$ has a triangular expansion into monomial symmetric functions, and we identify the leading terms in these decompositions. 
We also prove a version of a transition formula for affine fixed-point-free involutions. 

A notable aspect of this work, which did not apply to the cases considered in \cite{MZ},
is its connection to Rains and Vazirani's theory of \emph{quasiparabolic sets} from \cite{RV}. 
The motivating example of a quasiparabolic set is the set of left cosets of a parabolic subgroup $W_J$ in a Coxeter group $W$. Any quasiparabolic set inherits a Bruhat order and other nice properties, generalizing the features of this example.

We confirm a prediction of Rains and Vazirani \cite[Remark 4.4]{RV} that the 
set of affine fixed-point-free involutions naturally forms a quasiparabolic set for $\tS_n$.
This result is an important step in our proof of the affine fixed-point-free transition formula.
In particular, we use several results in \cite{RV} concerning the \emph{quasiparabolic Bruhat order}.
Since affine fixed-point-free involutions
form a quasiparabolic set, they inherit a module structure for the Iwahori-Hecke algebra of $\tS_n$.
Moreover, this module has a bar operator which leads to notions of canonical bases, $W$-graphs, and cells that 
are of independent interest. 

To conclude this introduction, we give an outline of the rest of this article. Section~\ref{pre-sect} introduces the affine symmetric group and Lam's affine Stanley symmetric functions. In Section~\ref{defn-sect}, we give the basic definitions of affine FPF involution Stanley symmetric functions and discuss their properties. In Section~\ref{qp-sect}, we prove that the 
set of affine fixed-point-free involutions forms a quasiparabolic set for $\tS_n$ and
discuss some applications and related open questions.
In Section~\ref{trans-sect}, finally, we prove our affine FPF involution transition formula. 

\subsection*{Acknowledgements}

I thank my PhD advisor, Eric Marberg, for guidance and helpful conversations.

\section{Preliminaries}\label{pre-sect}

Let $n$ be a positive integer. Write $\ZZ$ for the set of integers and define $[n] = \{1,2,\dots,n\}$. We also let $\NN=\{0,1,2,\cdots\}$.
The \emph{affine symmetric group} $\tS_n$ is the group of bijections $\pi:\ZZ \to \ZZ$
satisfying 
\[ \pi(i+n) = \pi(i)+n\text{ for all $i \in \ZZ$}
\qquand
\pi(1) + \pi(2) + \dots +\pi(n) = 1 + 2 + \dots +n.\]
Elements of $\tS_n$ are \emph{affine permutations}.
We will usually assume $n\ge2$ since  $\tS_1 = \{1\}$ is trivial.

Just as finite permutations can be represented as braid diagrams, affine permutations can be drawn as cylinder braid diagrams. Here is an example:
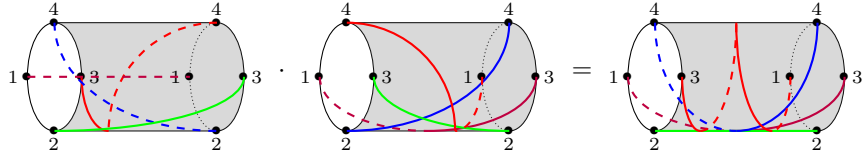
\[
\begin{tikzpicture}[baseline=0,scale=0.18,label/.style={postaction={ decorate,transform shape,decoration={ markings, mark=at position .5 with \node #1;}}}]
{
\draw (-6,0) circle (2.0 and 4.0);
\draw (6,-4) arc (-90:90:2.0 and 4.0);
\draw[densely dotted] (6,4) arc (90:270:2.0 and 4.0);
\draw (-6,4) -- (6,4);
\draw (-6,-4) -- (6,-4);
\fill [gray,opacity=0.30] (-6,4) -- (6,4) arc (90:-90:2.0 and 4.0) -- (-6,-4) arc (-90:90:2.0 and 4.0);
\node at (-8,0) {$_\bullet$};
\node at (-9,0) {$_1$};
\node at (-6,-4) {$_\bullet$};
\node at (-6,-5) {$_2$};
\node at (-4,0) {$_\bullet$};
\node at (-3,0) {$_3$};
\node at (-6,4) {$_\bullet$};
\node at (-6,5) {$_4$};
\node at (4,0) {$_\bullet$};
\node at (3,0) {$_1$};
\node at (6,-4) {$_\bullet$};
\node at (6,-5) {$_2$};
\node at (8,0) {$_\bullet$};
\node at (9,0) {$_3$};
\node at (6,4) {$_\bullet$};
\node at (6,5) {$_4$};
\draw [thick,red] (-4,0) arc (180:270:2 and 4);
\draw[thick,dashed, red] (6,4) arc (90:180:8 and 8);
\draw[thick,dashed, blue] (-6,4) arc (180:270:12 and 8);
\draw [thick,green] (-6,-4) arc (-90:0:14 and 4);
\draw[thick,dashed, purple] (-8,0) -- (4,0);
}
\end{tikzpicture}
\cdot
\begin{tikzpicture}[baseline=0,scale=0.18,label/.style={postaction={ decorate,transform shape,decoration={ markings, mark=at position .5 with \node #1;}}}]
{
\draw (-6,0) circle (2.0 and 4.0);
\draw (6,-4) arc (-90:90:2.0 and 4.0);
\draw[densely dotted] (6,4) arc (90:270:2.0 and 4.0);
\draw (-6,4) -- (6,4);
\draw (-6,-4) -- (6,-4);
\fill [gray,opacity=0.30] (-6,4) -- (6,4) arc (90:-90:2.0 and 4.0) -- (-6,-4) arc (-90:90:2.0 and 4.0);
\node at (-8,0) {$_\bullet$};
\node at (-9,0) {$_1$};
\node at (-6,-4) {$_\bullet$};
\node at (-6,-5) {$_2$};
\node at (-4,0) {$_\bullet$};
\node at (-3,0) {$_3$};
\node at (-6,4) {$_\bullet$};
\node at (-6,5) {$_4$};
\node at (4,0) {$_\bullet$};
\node at (3,0) {$_1$};
\node at (6,-4) {$_\bullet$};
\node at (6,-5) {$_2$};
\node at (8,0) {$_\bullet$};
\node at (9,0) {$_3$};
\node at (6,4) {$_\bullet$};
\node at (6,5) {$_4$};
\draw[thick,dashed,purple] (-8,0) arc (180:270:8 and 4);
\draw[thick,purple] (0,-4) arc (-90:0:8 and 4);
\draw[thick,blue] (-6,-4) arc (-90:0:12 and 8);
\draw[thick,green] (-4,0) arc (-180:-90:10 and 4);
\draw[thick,red,dashed] (4,0) arc (0:-90:2 and 4);
\draw[thick,red] (-6,4) arc (90:0:8 and 8);
}
\end{tikzpicture}
=
\begin{tikzpicture}[baseline=0,scale=0.18,label/.style={postaction={ decorate,transform shape,decoration={ markings, mark=at position .5 with \node #1;}}}]
{
\draw (-6,0) circle (2.0 and 4.0);
\draw (6,-4) arc (-90:90:2.0 and 4.0);
\draw[densely dotted] (6,4) arc (90:270:2.0 and 4.0);
\draw (-6,4) -- (6,4);
\draw (-6,-4) -- (6,-4);
\fill [gray,opacity=0.30] (-6,4) -- (6,4) arc (90:-90:2.0 and 4.0) -- (-6,-4) arc (-90:90:2.0 and 4.0);
\node at (-8,0) {$_\bullet$};
\node at (-9,0) {$_1$};
\node at (-6,-4) {$_\bullet$};
\node at (-6,-5) {$_2$};
\node at (-4,0) {$_\bullet$};
\node at (-3,0) {$_3$};
\node at (-6,4) {$_\bullet$};
\node at (-6,5) {$_4$};
\node at (4,0) {$_\bullet$};
\node at (3,0) {$_1$};
\node at (6,-4) {$_\bullet$};
\node at (6,-5) {$_2$};
\node at (8,0) {$_\bullet$};
\node at (9,0) {$_3$};
\node at (6,4) {$_\bullet$};
\node at (6,5) {$_4$};
\draw[thick,dashed,purple] (-8,0) arc (180:270:8 and 4);
\draw[thick,purple] (0,-4) arc (-90:0:8 and 4);
\draw[thick,green] (-6,-4) -- (6,-4);
\draw[thick,blue,dashed](-6,4) arc (180:270:6 and 8);
\draw[thick,blue](6,4) arc (0:-90:6 and 8);
\draw[thick,red](-4,0) arc (180:270:1.33333333333333333 and 4);
\draw[thick,red,dashed] (-2.666666666666667,-4) arc(-90:0:2.666666666666667 and 8);
\draw[thick,red] (0,4) arc(180:270:2.666666666666667 and 8);
\draw[thick,red,dashed](4,0) arc (0:-90:1.33333333333333333 and 4);
}
\end{tikzpicture}
\]
This represents the product $\pi \cdot \sigma$ of the permutations  $\pi,\sigma\in \tS_4$
that have $(\pi(1), \pi(2), \pi(3), \pi(4)) = (1,0,2,7)$ and 
$(\sigma(1), \sigma(2), \sigma(3), \sigma(4))=(4,3,1,2)$.
To read the diagram of $\pi$, start at $i$ on the right and trace the wire to its value on the left. Each time one crosses the top line going clockwise from the right, add $n$ to the value. Similarly, each time one crosses the line going counterclockwise viewing from the right, subtract $n$ from the value. The result gives $\pi(i)$.

Let $s_i $ for $i \in \ZZ$  be the unique element of $\tS_n$ that interchanges $i$ and $i+1$ while fixing every integer $j \notin \{i,i+1\} + n\ZZ$.
One has $s_i = s_{i+n}$ for all $i \in \ZZ$, and 
 $\{s_1,s_2,\dots,s_n\}$ generates the group $\tS_n$.
With respect to this generating set, $\tS_n$ is the affine Coxeter group of type $\tilde A_{n-1}$.
Let $S_n=\langle s_1,s_2,\dots,s_{n-1}\rangle$ denote the subgroup of permutations in $\tS_n$ that preserve the set $[n]$.

A \emph{reduced expression} for $\pi \in \tS_n$ is a minimal-length factorization $\pi = s_{i_1}s_{i_2}\cdots s_{i_l}$.
The \emph{length} of $\pi \in \tS_n$, denoted $\ell(\pi)$, is the number of factors  in any of its reduced expressions.
The value of $\ell(\pi)$ is also the number of equivalence classes in the set
$
\Inv(\pi) = \{ (i,j) \in \ZZ \times \ZZ : i< j\text{ and }\pi(i)> \pi(j)\}
$
under the relation $\sim$
on $\ZZ\times \ZZ$ with $(a,b) \sim (a',b')$ if and only if $a-a' = b-b' \in n \ZZ$.

To define the affine analogue of Stanley symmetric functions, we need to discuss \emph{cyclically decreasing expressions}.
A reduced expression $\pi = s_{i_1}s_{i_2}\cdots s_{i_l}$ for an affine permutation
is \emph{cyclically decreasing} if $s_{i_j + 1} \neq s_{i_k}$ for all $1 \leq j < k \leq l$. 
An element $\pi \in \tS_n$ is \emph{cyclically decreasing} if it has a cyclically decreasing reduced expression.

\begin{definition}[Lam \cite{Lam}] \label{lam-def}
The \emph{(affine) Stanley symmetric function} of $\pi \in \tS_n$ is 
\[ \tF_\pi = \sum_{\pi = \pi^1 \pi^2 \cdots } x_1^{\ell(\pi^1)} x_2^{\ell(\pi^2)} \cdots \in \ZZ[[x_1,x_2,\dots]]\]
where the sum is over all factorizations $\pi = \pi^1 \pi^2 \cdots$ 
of $\pi$
into countably many (possibly empty) cyclically decreasing factors
$\pi^i \in \tS_n$ such that $\ell(\pi) = \ell(\pi^1) + \ell(\pi^2) + \dots$.
\end{definition}

This definition is an extension of the symmetric functions introduced by Stanley in \cite{Stan}. If $\pi \in S_n \subsetneq \tS_n$ then $\tF_\pi$ 
 coincides with what is denoted $F_{\pi^{-1}}$ in \cite{Stan}. Our inverted indexing convention follows Lam \cite{Lam}.

\begin{example}
Suppose $n=2$ so that $s_1=s_3$. Then the only cyclically decreasing elements of $\tS_2$ are $1$, $s_1$, and $s_2$, so if $\pi \in \tS_2$ has length $\ell(\pi) = k$
then $\tF_\pi = \sum_{i_1<i_2< \dots<i_k} x_{i_1}x_{i_2}\cdots x_{i_k}= m_{1^k}$, where $m_\lambda$ denotes
the usual monomial symmetric function of a partition $\lambda$, i.e., 
the power series symmetrizing $x_1^{\lambda_1}\cdots x_n^{\lambda_n}$.
\end{example}
 
\begin{example}\label{stan-ex}
Suppose $n=4$ so that $s_1=s_5$.
There are four reduced expressions for the affine permutation $\pi = s_4s_2s_3s_1= s_2s_4s_3s_1=s_4s_2s_1s_3= s_2s_4s_1s_3 \in \tS_4$.
There are 11 cases for the distinct length-additive factorizations of this element.
Examples of cyclically decreasing factorizations for $\pi$ include  
\[ (s_4)(s_2)(s_3)(s_1) \qquad\text{and}\qquad (s_4s_2)(s_1)(s_3)=(s_2s_4)(s_1)(s_3).\] 
One can check that $\tF_{\pi} = 4m_{1111} + 2m_{211}+m_{22}$.
 \end{example}

%

Next, we discuss some useful combinatorial properties of affine permutations.

Fix an affine permutation $\pi \in \tS_n$.
The \emph{code} of $\pi$  is the sequence $c(\pi) = (c_1,c_2,\dots,c_n)$ where $c_i $ is the number of integers $j\in \ZZ$ with $ i <j$ and $\pi(i)>\pi(j)$.
An integer $i \in \ZZ$ is a \emph{descent} of $\pi$ if $\pi(i) > \pi(i+1)$, i.e., if $\ell(\pi s_i) = \ell(\pi) -1$.
This holds
if and only if $c_i > c_{i+1}$, taking $c_{n+1} = c_1$.
If $i \in [n]$ is a descent of $\pi$ then 
\be\label{ccc-eq}
c(\pi s_i) = (c_1,\dots,c_{i-1}, c_{i+1}, c_i - 1,c_{i+2},\dots, c_n),
\ee interpreting indices cyclically as necessary.
It holds that $|c(\pi)| := c_1 + c_2 + \dots c_n = \ell(\pi)$. 
The \emph{shape} $\lambda(\pi)$ of $\pi \in \tS_n$ is the transpose of the partition sorting $c(\pi^{-1})$.

A \emph{window} for an affine permutation $\pi \in \tS_n$ is a sequence of the form $[\pi(i+1), \pi(i+2),\dots,\pi(i+n)]$ where $i \in \ZZ$.
An element $\pi \in \tS_n$ is uniquely determined by any of its windows,
and a sequence of $n$ distinct integers is a window for some $\pi \in\tS_n$ if and only if the integers represent
each congruence class modulo $n$ exactly once. 

%
Write $<$ for the \emph{dominance order} on partitions, i.e., the partial order in which $\mu \leq \lambda$ if $\mu$ and $\lambda$ are partitions of $n$ such that
$\mu_1 + \dots +\mu_i \leq \lambda_1 + \dots + \lambda_i$ for all $i\geq 1$.

\begin{theorem}[{Lam \cite[Theorem 13]{Lam}}]\label{uni-thm}
If $\pi \in \tS_n$ then $\tF_\pi \in m_{\lambda(\pi)} + \sum_{\mu < \lambda(\pi)} \NN m_\nu$.
\end{theorem}

Let $\Par^n$ be the set of all partitions $\lambda$ with parts all at most $n-1$.
Let $\Sym^{(n)} = \QQ\spanning\{ m_\lambda : \lambda \in \Par^n\}$.
This theorem implies that 
$ \QQ\spanning\{ \tF_\pi : \pi \in \tS_n \}
= \Sym^{(n)}$.

\begin{example}\label{lambda-eg}
Suppose $n=4$ and $\pi = [-3,3,4,6]$. Then we have $c(\pi)=(0,1,1,2)$ and $c(\pi^{-1}) = (4,0,0,0)$.
So $\lambda(\pi) = (4)^T = (1,1,1,1)$ and $\lambda(\pi^{-1}) = (2,1,1)^T = (3,1)$.
Moreover, $\tF_\pi=m_{1111}$ and $\tF_{\pi^{-1}} = m_{31}+m_{22}+m_{211}+m_{1111}$.
\end{example}

Let $\DesR(\pi) = \{ s_i : i\in\ZZ\text{ is a descent of }\pi\} = \{ s \in \{ s_1,s_2,\dots,s_n\} : \ell(\pi s) < \ell(\pi)\}$
and $\DesL(\pi) = \DesR(\pi^{-1})$.
An element $\pi \in \tS_n$ is \emph{Grassmannian} if $\pi^{-1}(1)  < \pi^{-1}(2) < \dots < \pi^{-1}(n)$.
This occurs if and only if $\DesL(\pi) \subset \{ s_n\}$, or equivalently if $c(\pi^{-1})$ is weakly increasing.

\begin{definition}
The \emph{affine Schur function} $\tF_\lambda$ indexed by $\lambda \in \Par^n$
is the Stanley symmetric function
$\tF_\lambda = \tF_{\pi}$ where $\pi \in \tS_n$ is the unique Grassmannian element of shape $ \lambda$.
\end{definition}

Lam has shown that the symmetric functions $\tF_\pi$ are \emph{affine Schur positive} in the following sense:

\begin{theorem}[{Lam \cite[Corollary 8.5]{Lam2008}}]\label{+-thm}
$\NN\spanning\{ \tF_\pi : \pi \in \tS_n \} = \NN\spanning\{ \tF_\lambda : \lambda \in \Par^n\}.$
\end{theorem}

The Stanley symmetric functions indexed by $\pi \in S_n \subsetneq \tS_n$ have the following positivity property:

\begin{theorem}[See \cite{EG,Lascoux}]
$\NN\spanning\{ \tF_\pi : \pi \in S_n \} = \NN\spanning\{ s_\lambda : \lambda \in \Par^n,\ \lambda \subset (n-1,\dots,2,1)\}.$
\end{theorem}

Not all affine Schur functions are contained in $ \NN\spanning\{ s_\lambda : \lambda \in \Par^n\}$
so this theorem does not hold for arbitrary $\tF_\pi$.

Write $w \mapsto w^*$ for the unique group automorphism of $\tS_n$ with $s_i \mapsto s_i^* := s_{n-i}$ for $i \in \ZZ$.
If $\lambda \in \Par^n$ then there exists a unique Grassmannian permutation $\pi$ with $\lambda = \lambda(\pi)$,
and one defines $\lambda^* = \lambda(\pi^*)$.
Let $\lambda'(\pi)= \lambda(\pi^{-1})^*$ for $\pi \in \tS_n$
and define $<^*$ to be the partial order on $\Par^n$ with $\lambda <^* \mu$ if and only if $\mu^* < \lambda^*$.

\begin{example}
Suppose $n=4$ and $\pi =s_1s_2s_3s_4= [-3,3,4,6]$. Then $\pi^*=s_3s_2s_1s_4=[-1,1,2,8]$ and $c(\pi^*) = (0,0,0,4)$.
Since $\pi^{-1}$ is Grassmannian, we have $\lambda'(\pi) = \lambda((\pi^{-1})^*) =  \lambda((\pi^{*})^{-1}) = (4)^T = (1,1,1,1)$.
\end{example}

The next theorem shows that $\tF_\pi$ has a unitriangular expansion into $F_{\lambda}$'s.

\begin{theorem}[{Lam \cite{Lam}}]\label{schur-thm}
If $\pi \in \tS_n$ then 
$ \tF_\pi \in \( F_{\lambda'(\pi)} + \sum_{ \lambda'(\pi) <^* \mu} \NN \tF_\mu\) \cap \( \tF_{\lambda(\pi)} + \sum_{\mu < \lambda(\pi)} \NN \tF_\mu\).$
\end{theorem}

The affine Schur functions form a basis for $\Sym^{(n)}$, so there exists a unique linear involution $\omega^+ : \Sym^{(n)} \to \Sym^{(n)}$ with
$\omega^+(\tF_\lambda) = \tF_{\lambda^*}$ for all $\lambda \in \Par^n$.
This map can be defined directly in terms of the usual elementary, homogeneous, and monomial symmetric functions; see \cite[\S9]{Lam}.

\begin{theorem}[{Lam \cite[Theorem 15 and Proposition 17]{Lam}}]\label{inv-cong-thm}
 If $\pi \in \tS_n$ then $\omega^+(\tF_\pi) = \tF_{\pi^*} = \tF_{\pi^{-1}}$.
\end{theorem}

\section{Definitions and basic properties}\label{defn-sect}

Assume $n$ is a positive even integer. Define $\tI_n$ as the set of affine permutations $w\in\tS_n$
with $i=w(w(i))$ for all $i \in \ZZ$. We call these elements \emph{affine involutions}. Define $\cF_n$ as the set of affine involutions $w\in\tI_n$ with $w(i)\neq i$ for all $i \in \ZZ$. Elements of $\cF_n$ are \emph{affine fixed-point-free (FPF) involutions}. 

For integers $i<j\not\equiv i \modu n)$, define $t_{ij} \in \tS_n$ as the permutation exchanging $i+kn$ and $j+kn$ for all $k\in\ZZ$, and fixing all integers not in $\{i,j\} + n\ZZ$.
Viewing $\tS_n$ as a Coxeter group generated by $\{s_1,s_2,\dots,s_n\}$,
the set $\{ t_{ij} : i < j \not\equiv i \modu n)\}$ consists of all reflections in $\tS_n$.

We can use \emph{winding diagrams} to represent elements of $\cF_n$:
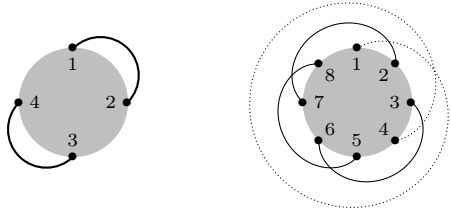
\[
\begin{tikzpicture}[baseline=0,scale=0.18,label/.style={postaction={ decorate,transform shape,decoration={ markings, mark=at position .5 with \node #1;}}}]
{
\draw[fill,lightgray] (-12,0) circle (4.0);
\node at (-12, 4.0) {$_\bullet$};
\node at (-12, 2.8) {$_1$};
\node at (-16, 0) {$_\bullet$};
\node at (-14.8, 0) {$_4$};
\node at (-12, -4) {$_\bullet$};
\node at (-12, -2.8) {$_3$};
\node at (-8.0, 0.0) {$_\bullet$};
\node at (-9.2, 0.0) {$_2$};
\draw[thick] (-12, 4.0) arc (135:-45:2.8284);
\draw[thick] (-12, -4.0) arc (-45:-225:2.8284);
}
\end{tikzpicture}
\qquad\qquad
\begin{tikzpicture}[baseline=0,scale=0.18,label/.style={postaction={ decorate,transform shape,decoration={ markings, mark=at position .5 with \node #1;}}}]
{
\draw[fill,lightgray] (0,0) circle (4.0);
\node at (2.44929359829e-16, 4.0) {$_\bullet$};
\node at (1.71450551881e-16, 2.8) {$_1$};
\node at (2.82842712475, 2.82842712475) {$_\bullet$};
\node at (1.97989898732, 1.97989898732) {$_2$};
\node at (4.0, 0.0) {$_\bullet$};
\node at (2.8, 0.0) {$_3$};
\node at (2.82842712475, -2.82842712475) {$_\bullet$};
\node at (1.97989898732, -1.97989898732) {$_4$};
\node at (2.44929359829e-16, -4.0) {$_\bullet$};
\node at (1.71450551881e-16, -2.8) {$_5$};
\node at (-2.82842712475, -2.82842712475) {$_\bullet$};
\node at (-1.97989898732, -1.97989898732) {$_6$};
\node at (-4.0, -4.89858719659e-16) {$_\bullet$};
\node at (-2.8, -3.42901103761e-16) {$_7$};
\node at (-2.82842712475, 2.82842712475) {$_\bullet$};
\node at (-1.97989898732, 1.97989898732) {$_8$};
\draw [-,>=latex,domain=0:100,samples=100,densely dotted] plot ({(4.0 + 4.0 * sin(180 * (0.5 + asin(-0.9 + 1.8 * (\x / 100)) / asin(0.9) / 2))) * cos(90 - (0.0 + \x * 4.95))}, {(4.0 + 4.0 * sin(180 * (0.5 + asin(-0.9 + 1.8 * (\x / 100)) / asin(0.9) / 2))) * sin(90 - (0.0 + \x * 4.95))});
\draw [-,>=latex,domain=0:100,samples=100] plot ({(4.0 + 2.0 * sin(180 * (0.5 + asin(-0.9 + 1.8 * (\x / 100)) / asin(0.9) / 2))) * cos(90 - (90.0 + \x * 1.35))}, {(4.0 + 2.0 * sin(180 * (0.5 + asin(-0.9 + 1.8 * (\x / 100)) / asin(0.9) / 2))) * sin(90 - (90.0 + \x * 1.35))});
\draw [-,>=latex,domain=0:100,samples=100] plot ({(4.0 + 2.0 * sin(180 * (0.5 + asin(-0.9 + 1.8 * (\x / 100)) / asin(0.9) / 2))) * cos(90 - (270.0 + \x * 1.35))}, {(4.0 + 2.0 * sin(180 * (0.5 + asin(-0.9 + 1.8 * (\x / 100)) / asin(0.9) / 2))) * sin(90 - (270.0 + \x * 1.35))});
\draw [-,>=latex,domain=0:100,samples=100] plot ({(4.0 + 2.0 * sin(180 * (0.5 + asin(-0.9 + 1.8 * (\x / 100)) / asin(0.9) / 2))) * cos(90 - (180.0 + \x * 1.35))}, {(4.0 + 2.0 * sin(180 * (0.5 + asin(-0.9 + 1.8 * (\x / 100)) / asin(0.9) / 2))) * sin(90 - (180.0 + \x * 1.35))});
}
\end{tikzpicture}
\]
In these diagrams, the numbers $1,2,\dots,n$ are arranged in order around a circle,
and every number $i$ is connected to another number $j$ by a path the winds around the circle in a clockwise direction. A curve that goes from $i$ to $j$ and travels
$m_{ij}$ times past the vertex 1 represents the cycle $t_{i,j+mn}$.
(These curves are drawn in different styles for readability.)
The left example shows $y= t_{1,2}t_{3,4}\in\cF_4$, while the 
right example shows $z= t_{1,12} \cdot t_{3,6}\cdot t_{5,8}\cdot t_{7,10}  \in \cF_8$.
Since none of the curves in the diagram for $y$ go past the point 1, we have $y \in S_4 \subsetneq \tS_4$.

\begin{definition}
Given $\pi \in\tS_n$, define $\beta(\pi)=\frac{1}{2n}\sum_{i=1}^n|\pi(i)-r_n(\pi(i))|,$ where $r_n(i)$ for $i \in \ZZ$ denotes the unique element of $[n]$ that satisfies $r_n(i)\equiv i\modu{n})$. For $z \in \cF_n$, define $\sfpf(z)=(-1)^{\beta(z)}$.
\end{definition}


\begin{lemma}\label{beta-inv-lem}
Let $z \in \cF_n$ and $l=n/2$. 
 Suppose $a_1<a_2<\dots <a_l$ are the numbers $a \in  [n]$ with $a < z(a)$ and $b_i=z(a_i)$.
Then $\beta(z)=\frac{1}{n}\sum_{i=1}^l(a_i+b_i)-\frac{n+1}{2}$.
\end{lemma}

\begin{proof}
We have 
 $z=t_{a_1,b_1}\cdots t_{a_l,b_l}$ and $\beta(z)=\beta(t_{a_1,b_1})+\cdots+\beta(t_{a_l,b_l})$. 
 Define $h_i = b_i - r_n(b_i)\geq 0$. Then 
 \[\beta(t_{a_i,b_i})=\tfrac{1}{2n}(|b_i-r_n(b_i)|+|t_{a_i,b_i}(r_n(b_i))-r_n(t_{a_i,b_i}(r_n(b_i)))|)=\tfrac{1}{2n}(h_i+|a_i-h_i-a_i|)=\tfrac{h_i}{n}.\]
So we have $\beta(z)=\frac{1}{n}\sum_{i=1}^l h_i=\frac{1}{n}\sum_{i=1}^l(a_i+b_i)-\frac{1}{n}\sum_{i=1}^l(a_i+r_n(b_i))=\frac{1}{n}\sum_{i=1}^l(a_i+b_i)-\frac{n+1}{2}$.
\end{proof}


\begin{lemma}\label{par-inv-lem}
If $z \in \cF_n$ and $w \in \tS_n$ then $\sfpf(wzw^{-1}) = \sfpf(z)$.
\end{lemma}

\begin{proof}
Since $\tS_n$ is generated by $s_1,\ldots,s_n$, it suffices to prove that $\sfpf(s_izs_i) = \sfpf(z)$ for all $i \in [n]$.  Note that we can write $z=t_{a_1,b_1}\cdots t_{a_l,b_l}$ with $a_i < b_i=z(a_i)$ defined as above. Let $y=s_izs_i$.
Define 
\[\Cyc(z)=\{(p,z(p)) : p<z(p)\text{ and }p\in[n]\}=\{(a_i,b_i) : 1\le i\le l\}.\] 
Write $\sim$ for the equivalence relation on $\ZZ\times \ZZ$ with $(x,y) \sim (x+n,y+n)$ and
let 
$\wCyc(z)=\Cyc(z)/\sim$.
Denote the equivalence class of $(p,z(p))$ in $\wCyc(z)$ as $\overline{(p,z(p))}$.

If $s_i=t_{a_j,b_j}$, then we have $y=z$ and $\sfpf(y)=\sfpf(z)$. 
If $y\neq z$, then we suppose $a'_1<a'_2<\dots <a'_l$ are the numbers $a' \in  [n]$ with $a' < y(a')$ and $b'_i=y(a'_i)$. There are four cases to consider:
\begin{itemize}

\item Suppose $a_j=i$ and $a_k\equiv i+1\modu{n})$. Then $\{\overline{(a_j+1,b_j)},\overline{(a_k-1,b_k)}\}\subset\wCyc(y)$. If $a_j+1\le n$ then 
$\{(a_j+1,b_j),(a_k-1,b_k)\}\subset \Cyc(y)$ while the other cycles remain unchanged. Therefore we have 
\be\label{abab-eq}
\sum_{i=1}^n(a'_i+b'_i)=\sum_{i=1}^n(a_i+b_i)
\ee
 and $\sfpf(y)=\sfpf(z)$. If $a_j=n$, then 
$\{(a_j+1-n,b_j-n),(a_k-1+n,b_k+n)\}\subset \Cyc(y)$ while the other cycles remain unchanged. Therefore
\eqref{abab-eq} holds again
 and $\sfpf(y)=\sfpf(z)$.

\item Suppose $a_j=i$ and $b_k\equiv i+1\modu{n})$. Then $\{\overline{(a_j+1,b_j)},\overline{(a_k,b_k-1)}\}\subset\wCyc(y).$ If $a_j+1\le n$ and $a_k=a_j$, then $\{(a_j+1,b_j-1)\}\subset\Cyc(y)$ while the other cycles remain unchanged. 
 If $a_j+1\le n$ and $a_k<a_j$, then 
$\{(a_j+1,b_j),(a_k,b_k-1)\}\subset\Cyc(y)$ while the other cycles remain unchanged. 
If $a_j=n$, finally, then 
$\{(a_j+1-n,b_j-n),(a_k,b_k-1)\}\subset\Cyc(y)$ while the other cycles remain unchanged. 
In each case
\eqref{abab-eq} holds again
and $\sfpf(y)=\sfpf(z)$.

\item Suppose $b_j\equiv i\modu{n})$ and $a_k\equiv i+1\modu{n})$. Then $\{\overline{(a_j,b_j+1)},\overline{(a_k-1,b_k)}\}\subset\wCyc(y).$ If $a_k-1>0$ then $\{(a_j,b_j+1),(a_k-1,b_k)\}\subset\Cyc(y)$ while the other cycles remain unchanged. 
Therefore 
\eqref{abab-eq} holds 
and $\sfpf(y)=\sfpf(z)$. 

If $a_k=1$, then 
$\{(a_j,b_j+1),(a_k-1+n,b_k+n)\}\subset\Cyc(y)$ while the other cycles remain unchanged. Therefore
we have $\sum_{i=1}^n(a'_i+b'_i)=\sum_{i=1}^n(a_i+b_i)+2n$ and
  $\sfpf(y)=\sfpf(z)$.

\item Suppose $b_j\equiv i\modu{n})$ and $b_k\equiv i+1\modu{n})$. Then $\{(a_j,b_j+1),(a_k,b_k-1)\}\subset\Cyc(y)$ while the other cycles remain unchanged. Therefore 
\eqref{abab-eq} holds and 
 $\sfpf(y)=\sfpf(z)$.
\end{itemize}
\end{proof}

Let $\tp =s_1s_3\cdots s_{n-1}= [2,1,4,3,\dots,n,n-1] \in \tI_n$ and 
$\Theta^-= s_2s_4\cdots s_n= [1,0,3,2,\dots,n-1,n-2]  \in \tI_n$,
so that $\sfpf(\Theta^\pm)=\pm 1$.
We reserve the symbol $\Theta$ to denote one of the two elements of $\{ \Theta^\pm\}$.
Define $\cFp$ as the $\tS_n$-conjugacy class of $\Theta^+$ and $\cFm$ as the $\tS_n$-conjugacy class of $\Theta^-$. One can show that
\be\label{+-eq} \cFp =\{z\in\cF_n : \sfpf(z)=1\} \qquand
 \cFm =\{z\in\cF_n : \sfpf(z)=-1\}\ee
and hence that $\cF_n=\cFp\sqcup\cFm$; see, e.g., \cite[Theorem 5.4]{Mar}.
For $z \in \cFp$, let $\cAfpf(z)$ be the set of minimal length elements in  $\{ w \in \tS_n : w^{-1}\Theta^+ w = z\}$.
Similarly, $z \in \cFm$, let $\cAfpf(z)$ be the set of minimal length elements in  $\{ w \in \tS_n : w^{-1}\Theta^-  w = z\}$.
We refer to elements of $\cAfpf(z)$ as \emph{FPF atoms} for $z$ .

\begin{definition}\label{aff-def}
The \emph{(affine) FPF-involution Stanley symmetric function} of $z \in \cF_n$ is
$$\Ffpf_z = \sum_{\pi \in \cAfpf(z)} \tF_\pi \in \Sym^{(n)}.$$
\end{definition}

This definition is natural to consider in view of \cite{Lam} and \cite{MZ}. 
The power series $\Ffpf_z$ for $z \in \cF_n \cap S_n$ were previously studied in \cite{HMP5}.

\begin{example}
If $n=2$ then every $z \in \cF_2$ has only one FPF-atom of length $k=\frac{\ell(z)-1}{2}$, so $\Ffpf_z = m_{1^k}$.
\end{example}

For $z \in \cF_n$ and $l=n/2$, again define $a_1<a_2<\dots <a_l$ to be the numbers $a \in  [n]$ with $a < z(a)$. Similarly, let $d_1<d_2<\dots <d_l$ be the numbers $d \in [n]$ with $z(d) < d$.
Set $b_i = z(a_i)$ and $c_i = z(d_i)$.
Define 
\begin{equation}\label{amin-def}
    \alpha^\fpf_{\min}(z)=[a_1,b_1,a_2, b_2,\dots,a_l,b_l]^{-1}
\qquand
    \alpha^\fpf_{\max}(z)=[c_1,d_1,c_2,d_2,\dots,c_l,d_l]^{-1}
    \end{equation}
and let $\ellfpf(z) = \tfrac{1}{2}(\ell(z) - \tfrac{n}{2})$.

\begin{theorem}\label{length-alphamin-thm}
If $z \in \cF_n$ then $\ellfpf(z)=\ell(w)$ for all $w\in\cAfpf(z)$, and $\left\{\alpha^\fpf_{\min}(z),\alpha^\fpf_{\max}(z)\right\}\subset \cAfpf(z)$.
\end{theorem}

The elements $\alpha^\fpf_{\min}(z)$ and $\alpha^\fpf_{\max}(z)$ are 
denoted $\alpha_R(z)$ and $\alpha_L(z)$ in \cite[Theorem 5.12]{Mar}.
One can derive the theorem from results in \cite[\S5]{Mar}, but we give a direct, self-contained proof.

\begin{proof}
Let $\Theta \in \{\Theta^\pm\}$ be the element conjugate to $z$. It is obvious that 
$\ell(z)=\ell(w^{-1}\Theta w)\le\ell(w^{-1})+\ell(\Theta)+\ell(w)=2\ell(w)+\frac{n}{2}.$
Therefore, if $w\in\cAfpf(z)$, then $\ell(w)\ge\ellfpf(z)$.
Let $\alpha = \alpha^\fpf_{\min}(z)$. We prove that $\alpha^{-1}\Theta\alpha=z$.
For $i \in \ZZ$ define $a_i=a_{r_l(i)}+2(i-r_l(i))$ where $l=\frac{n}{2}$ and $b_i=z(a_i)$.
One can compute that $\alpha(a_i)=2i+\beta(z)-1$ and $\alpha(b_i)=2i+\beta(z)$. 
If $\sfpf(z)=1$ then
\[
\alpha^{-1}\Theta^+\alpha=t_{\alpha^{-1}(1),\alpha^{-1}(2)}\cdots t_{\alpha^{-1}(n-1),\alpha^{-1}(n)}
=t_{a_{1-{\beta(z)}/{2}},b_{1-{\beta(z)}/{2}}}\cdots t_{a_{l-{\beta(z)}/{2}},b_{l-{\beta(z)}/{2}}}
=t_{a_1,b_1}\cdots t_{a_l,b_l}=z.
\]
The result when $\sfpf(z)=-1$ is similar. Finally, we compute lengths. We have
\begin{align*}
    \ell(\alpha)&=\ell([a_1,b_1,\cdots,a_l,b_l])
    =\sum_{i=1}^{l}(\#\{j>i:a_j<b_i\}+\#\{j>i:b_j<b_i\})\\
    &=\#\{(a_i,b_j):a_i\in[n],a_i<b_j,b_i>a_j\}+\#\{(b_j,b_i):b_j\in[n],b_j<b_i,a_i<a_j\}\\
    &=\tfrac{1}{2}\#\{(i,j):i\in[n],i<j,z(i)>z(j)\neq i\}
    =\tfrac{1}{2}(\ell(z)-\tfrac{n}{2}).
\end{align*}
Therefore we have $\alpha\in\cAfpf(z)$ and so it follows that all $w \in \cAfpf(z)$ must have $\ell(w) =\ellfpf(z)$. It follows similarly that $\alpha^\fpf_{\max}(z)\in\cAfpf(z)$.
\end{proof}
\begin{example}\label{aff-ex1}
Suppose $n=4$ and
\[z= 
\begin{tikzpicture}[baseline=0,scale=0.15,label/.style={postaction={ decorate,transform shape,decoration={ markings, mark=at position .5 with \node #1;}}}]
{
\draw[fill,lightgray] (0,0) circle (4.0);
\node at (2.4492935982947064e-16, 4.0) {$_\bullet$};
\node at (1.8369701987210297e-16, 3.0) {$_{1}$};
\node at (4.0, 0.0) {$_\bullet$};
\node at (3.0, 0.0) {$_{2}$};
\node at (2.4492935982947064e-16, -4.0) {$_\bullet$};
\node at (1.8369701987210297e-16, -3.0) {$_{3}$};
\node at (-4.0, -4.898587196589413e-16) {$_\bullet$};
\node at (-3.0, -3.6739403974420594e-16) {$_{4}$};
\draw [-,>=latex,domain=0:100,samples=100] plot ({(4.0 + 2.0 * sin(180 * (0.5 + asin(-0.9 + 1.8 * (\x / 100)) / asin(0.9) / 2))) * cos(90 - (180.0 + \x * 4.5))}, {(4.0 + 2.0 * sin(180 * (0.5 + asin(-0.9 + 1.8 * (\x / 100)) / asin(0.9) / 2))) * sin(90 - (180.0 + \x * 4.5))});
\draw [-,>=latex,domain=0:100,samples=100,densely dotted] plot ({(4.0 + 2.0 * sin(180 * (0.5 + asin(-0.9 + 1.8 * (\x / 100)) / asin(0.9) / 2))) * cos(90 - (\x * 4.5))}, {(4.0 + 2.0 * sin(180 * (0.5 + asin(-0.9 + 1.8 * (\x / 100)) / asin(0.9) / 2))) * sin(90 - (\x * 4.5))});}
\end{tikzpicture}
= t_{1,6}t_{3,8} = [6,-3,8,-1] \in \cF_4.
\]
We have $\beta(z)=2$, so $z\in\cFp$.
The elements of $\cAfpf(z)$ are
\[\ba \alpha^\fpf_{\min}(z) &= [1,6,3,8]^{-1} =[3,0,5,2] =[-3,2,-1,4]^{-1}=\alpha^\fpf_{\max}(z).
\ea\]
The permutation $[3,0,5,2] = s_4s_2s_3s_1$ has a single reduced expression,
and it holds that $\tF_{[3,0,5,2]} = 4m_{1111}+2m_{211}+m_{22}$.
Therefore $\Ffpf_z = \Ffpf_{[6,-3,8,-1]} = 4m_{1111}+2m_{211}+m_{22}.$
\end{example}

Next, we discuss some length formulas.

\begin{lemma}\label{t-length-lem}
For $i<j\not\equiv i\modu{n})$ and $w \in \tS_n$, it holds that $$\ell(wt_{ij})=\begin{cases}\ell(w)+2\delta(w,i,j)+1&\text{ if }w(i)<w(j)\\\ell(w)-2\delta(wt_{ij},i,j)-1&\text{ if }w(i)>w(j),\end{cases}$$ where 
$\delta(w,i,j)=|\{k \in \ZZ : i<k<j\text{ and }k\not\equiv i\modu{n})\text{ and }w(k)\text{ is between }w(i)\text{ and }w(j)\}|.$
\end{lemma}

We expect that this lemma may be known. We include a proof since we have not found a good reference.

\begin{proof}
Assume $w(i)<w(j)$. We count the pairs $(p,q)\in\ZZ\times\ZZ$ with $p\not\equiv q\modu{n})$ 
that are in the inversion set of $w$ and $wt_{ij}$. For pairs $(p,q)$ with $\{p,q\}\cap\{i,j\}+n\ZZ=\varnothing$, we have $w(p)<w(q)$ if and only if $wt_{ij}(p)<wt_{ij}(q)$ and $w(p)>w(q)$ if and only if $wt_{ij}(p)>wt_{ij}(q)$. So either $(p,q)$ is an inversion of both $w$ and $wt_{ij}$ or it is not an inversion of either $w$ or $wt_{ij}$. We only need to consider pairs of the form $(i,k),(j,k),(k,i),(k,j)$. Only the following cases contribute to the difference of $\ell(wt_{ij})$ and $\ell(w)$:
\begin{enumerate}

\item[1.] For $k\not\in\{i,j\}+n\ZZ$ with $i<k<j$ and $w(i)<w(k)<w(j)$, then $(i,k),(k,j) \in \Inv(wt_{ij})\setminus\Inv(w)$.
\item[2.] For $k\equiv j\modu{n})$, if $k<i<j$ and $w(i)<w(k)<w(j)$, then $(k,i) \in \Inv(wt_{ij})\setminus\Inv(w)$.
\item[3.] For $k\equiv j\modu{n})$, if $i<k<j$ and $w(k)<w(i)<w(j)$, then $(i,k)\in \Inv(w)\setminus\Inv(wt_{ij})$.
\item[4.] For $k\equiv j\modu{n})$ with $i<k<j$ and $w(i)<w(k)<w(j)$, we have $(i,k) \in \Inv(wt_{ij})\setminus\Inv(w)$.
\item[5.] For $k\equiv i\modu{n})$, if $i<j<k$ and $w(i)<w(k)<w(j)$, then $(j,k) \in \Inv(wt_{ij})\setminus\Inv(w)$.
\item[6.] For $k\equiv i\modu{n})$, if $i<k<j$ and $w(i)<w(j)<w(k)$, then $(k,j) \in \Inv(w)\setminus\Inv(wt_{ij})$.
\item[7.] For $k\equiv i\modu{n})$ with $i<k<j$ and $w(i)<w(k)<w(j)$, we have $(k,j) \in \Inv(wt_{ij})\setminus\Inv(w)$.
\item[8.] For pair $(i,j)$, we see that $(i,j) \in \Inv(wt_{ij})\setminus\Inv(w)$.
\end{enumerate}
Each of these cases contributes $1$ to the difference $\ell(wt_{ij})-\ell(w)$ except that case 1 contributes $2$. However, we find that cases 2 and 3 are in one-to-one correspondence with cases 5 and 6 through the map $k\mapsto k'=i+j-k$. Also, the pairs $(i,k)$ in case 4 are in one-to-one correspondence with the pairs $(w(k),j)$ in case 7. Thus, we have $\ell(wt_{ij})=\ell(w)+2\delta(w,i,j)+1$.

If $w(i)>w(j)$, then 
the result just shown implies that $\ell(w)=\ell(wt_{ij}t_{ij})=\ell(wt_{ij})+2\delta(wt_{ij},i,j)+1$.
\end{proof}

\begin{theorem}\label{atom-thm}
Fix an affine FPF-involution $z\in\cF_n$.
Then the set $\cAfpf(z)$ is a bounded poset, with unique minimum $\alpha^\fpf_{\min}(z)$ and unique maximum $\alpha^\fpf_{\max}(z)$, under the transitive relation on $\tS_n$ generated by
\[ [\cdots, a, d,b,c,\cdots]^{-1} \prec [\cdots,b,c,a,d,\cdots]^{-1}\]
if $a<b<c<d$ and the parts of each window masked by $\cdots$ are the same on either side of the relation.
\end{theorem}
This result is an affine generalization of \cite[Theorem 6.22]{HMP2}.

\begin{proof}
We again define $a_i < b_i = z(a_i)$ as in Lemma~\ref{beta-inv-lem}. 
Let $\Theta \in \{\Theta^\pm\}$ be the element conjugate to $z$.
First, we show that $\alpha^\fpf_{\min}(z)$ is the unique minimum of $\cAfpf(z)$. This element is minimal under $\prec$ since we cannot express
$[a_1,b_1,a_2,b_2,\cdots,a_l,b_l]^{-1}$ in the form $[\cdots,b,c,a,d,\cdots]^{-1}$ for $a<b<c<d$. 
Suppose $w \in \cAfpf(z)$ is also minimal under $\prec$. Since $w^{-1}\Theta w=z$, we must have $w^{-1}(a_i)=w^{-1}(b_i)\pm1$. Since $w$ is an FPF atom, we must have $w^{-1}(a_i)=w^{-1}(b_i)-1$. Therefore $w=[\cdots,a_i,b_i,a_j,b_j,\cdots]^{-1}$. Since $w$ is a minimal element under our relation, we cannot have $a_i>a_j$ and $b_i<b_j$. Since $w$ is an FPF atom, we also cannot have $a_i>a_j$ and $b_i>b_j$, since then the element $w'=wt_{w^{-1}(a_i),w^{-1}(a_j)}t_{w^{-1}(b_i),w^{-1}(b_j)}$
would have $w'^{-1}\Theta w'=z$ and $\ell(w')<\ell(w)$.
So we must have $a_i<a_j$ for all such $i,j$ and we get $w=\alpha^\fpf_{\min}(z)$. The argument to show that $\alpha^\fpf_{\max}(z)$ is the unique maximum of $\cAfpf(z)$ is similar.

Next we show that if $\alpha^\fpf_{\min}(z)\prec w\prec\alpha^\fpf_{\max}(z)$, then $w\in\cAfpf(z)$. We prove this inductively.
If $\alpha^\fpf_{\min}(z) =[\cdots, a,d,b,c,\cdots]^{-1}$ or $\alpha^\fpf_{\max}(z) =[\cdots,b,c,a,d,\cdots]^{-1}$ where $a<b<c<d$, then we have $z(a)=d$ and $z(b)=c$ by definition. 
We show that if $[\cdots,a,d,b,c,\cdots]^{-1}=w\prec v=[\cdots,b,c,a,d,\cdots]^{-1}$ with $z(a)=d$ and $z(b)=c$ and $a<b<c<d$, then $w\in\cAfpf(z)$ if and only if $v\in\cAfpf(z)$.

First, we show that in this case $w^{-1}\Theta w=z$ if and only if $v^{-1}\Theta v=z$. If $w^{-1}\Theta w=z$ and $w(d) = i$, then we have $z(a)=w^{-1}\Theta w(a)=w^{-1}\Theta(i-1)=w^{-1}(\Theta(i-1))=d$, so we have $\Theta(i-1)=i$. This implies that $v^{-1}\Theta v(a)=v^{-1}\Theta(i+1)=v^{-1}(i+2)=d$ and similarly $v^{-1}\Theta v(b)=c$. Thus $v^{-1}\Theta v=z$. The converse implication follows similarly.
Second, we show that the length remains unchanged under the relation $\prec$. 
This follows by Lemma~\ref{t-length-lem}, which implies that
$\ell(v)=\ell(wt_{a,b}t_{c,d})=\ell(wt_{a,b})-1=\ell(w).$
\end{proof}

\begin{lemma}\label{inv-set-lem}
A subset $I$ of $\ZZ\times\ZZ$ is the inversion set of an affine permutation $w\in\tS_n$ if and only if the following conditions hold:
\begin{enumerate}
\item[(1)] For all $(i,j)\in I$, we have $i<j\not\equiv i\modu{n})$.
\item[(2)] If $(i,j)\in I$, then $(i+n,j+n)\in I$ and $(i-n,j-n)\in I$.
\item[(3)] If $\{(i,j),(j,k)\}\subset I$, then $(i,k)\in I$.
\item[(4)] If $(i,k)\in I$, then for every $i<j<k$, either $(i,j)\in I$ or $(j,k)\in I$.
\end{enumerate}
Moreover, if $I \subset \ZZ\times \ZZ$ satisfies these conditions then $I = \Inv(w)$ for exactly one $w \in \tS_n$.
\end{lemma}
\begin{proof}
We first prove that every inversion set $\Inv(w)$ for $w \in \tS_n$ satisfies these conditions. The first two conditions obviously hold. If $\{(i,j),(j,k)\}\subset\Inv(w)$, then $w(i)>w(j)>w(k)$ and $(i,k)\in\Inv(w)$. If $(i,k)\in\Inv(w)$, then for every $i<j<k$, either $w(k)<w(i)<w(j)$, $w(j)<w(k)<w(i)$ or $w(k)<w(j)<w(i)$, which leads to the required result.

Conversely, suppose $I$ satisfies these conditions. We prove that $I = \inv(w)$ for a unique $w \in \tS_n$ by induction on the size of 
the quotient $I/ \sim$, where $\sim$ is the equivalence relation on $\ZZ\times \ZZ$ generated by $(i,j)\sim(i+n,j+n)$. If $I$ is empty then $I = \inv(1)$. If $I$ is nonempty then we claim that $I$ contains some pair of form $(i,i+1)$. If $I$ contains no pair of form $(i,i+1)$, then it contains no pair of form $(i,i+2)$, by condition (4). Recursively, $I$ then contains no pair of form $(i,i+k)$ for any $k\in\ZZ$, so is empty.

Suppose $(i,i+1)\in I$ and define $J = \{ (s_i(a), s_i(b)) : (a,b) \in I,\ (a,b) \notin (i,i+1) +n(\ZZ,\ZZ)\}$. Then $|J/\sim|=|I/\sim|-1$. We show that $J$ satisfies the given conditions. Obviously $J$ satisfies conditions (1) and (2). For condition (3), note that if $(s_i(a),s_i(b)),(s_i(b),s_i(c))\in J$, then $(a,b),(b,c)\in I$ and $(a,c)\in I$. Since $a<b<c$, we cannot have $(a,c)\in(i,i+1)+n\ZZ$. Hence $(s_i(a),s_i(c))\in J$. For condition (4), note that if $(s_i(a),s_i(c))\in J$, then $(a,c)\in I$ and for every $a<b<c$, either $(a,b)\in I$ or $(b,c)\in I$. As a result, for $s_i(a)<x<s_i(c)$, we have $s_i(x)=b$ for some $a<b<c$ unless $(s_i(a),x)\in(i,i+1)+n\ZZ$ or $(x,s_i(b))\in(i,i+1)+n\ZZ$. In the first case, $(a,c)\sim(i+1,c')\in I$, so $(i,c')\in I$ and $(x,c)\in J$. In the second case, $(a,c)\sim(a',i)\in I$, so $(a',i+1)\in I$ and $(a,x)\in J$. Therefore, $J$ satisfies our four conditions so by induction hypothesis, we have $J = \inv(v)$ for a unique $v\in\tS_n$. But it is easy to see that $J = \inv(v)$ if and only if $I =\inv(w)$ for the element $w=vs_i$.
Since $v$ is unique, $w$ must also be unique.
%
%
%
\end{proof}

\begin{proposition}
For an affine FPF-involution $z\in\cF_n$, the poset $(\cAfpf(z),\prec)$ is a lattice.
\end{proposition}

\begin{proof}
Let $l=\frac{n}{2}$ and $u\in\cAfpf(z)$.
If $z\in\cFp$ (respectively, $z \in \cFm$) then define
\[e_i = u^{-1}(2i-1)\text{ and } f_i = u^{-1}(2i)
\qquad\text{(respectively, 
$e_i = u^{-1}(2i) \text{ and } f_i= u^{-1}(2i+1)$)}\] for each $i \in \ZZ$. Then $u=[e_1,f_1,\cdots,e_{l},f_{l}]^{-1}$ and by Theorem~\ref{atom-thm} we have
 $e_i<f_i= z(e_i)$. By Theorem~\ref{atom-thm}, if $i<j$ then we cannot have $e_i>e_j$ and $f_i>f_j$. Define 
 \[\Inv^\fpf(u)=\{(i,j):i<j\not\equiv i\modu{n}),e_j<e_i<f_i<f_j\}.\]
We claim that this is the inversion set of a unique permutation $\pi\in\tS_{l}$. It satisfies the first two conditions in Lemma~\ref{inv-set-lem} obviously. For the third condition, note that if $\{(i,j),(j,k)\}\subset\Inv^\fpf(z)$, then we have $e_k<e_j<e_i<f_i<f_j<f_k$, so $(i,k)\in\Inv^\fpf(z)$. For the last condition, note that if $(i,k)\in\Inv^\fpf(z)$, then for each $j$ with $i<j<k$, one of the following cases occurs:
\begin{enumerate}
\item[(1)] $e_j<e_k<e_i<f_i<f_j<f_k$;
\item[(2)] $e_j<e_k<e_i<f_i<f_k<f_j$;
\item[(3)] $e_k<e_j<e_i<f_i<f_j<f_k$;
\item[(4)] $e_k<e_i<e_j<f_i<f_j<f_k$;
\item[(5)] $e_k<e_i<e_j<f_j<f_i<f_k$.
\end{enumerate}
Cases (1)-(3) imply that $(i,j)\in\Inv^\fpf(z)$ while cases (3)-(6) imply that $(j,k)\in\Inv^\fpf(z)$. 
We conclude by Lemma~\ref{inv-set-lem} that there is a unique $\pi\in\tS_{l}$ with $\Inv(\pi)=\Inv^\fpf(u)$.

Now consider the map $\phi : \cAfpf(z) \to \tS_l$ given by $\phi(u) = \pi$.
Recall from \cite[Proposition 3.1.2]{BB} that the right weak order on a Coxeter group with length function $\ell$ is
defined by setting $x\le_R y$ if and only if $\ell(y)=\ell(x)+\ell(x^{-1}y)$. 
It is easy to see that $\phi$ is an order-preserving bijection from $(\cAfpf(z),\prec)$ to an interval in $(\tS_l,<_R)$.
By the remarks after \cite[Theorem 3.2.1]{BB}, the right weak order on a Coxeter group with an upper bound is a lattice. Thus  $\phi$ is an isomorphism from $\cAfpf(z)$ to a sublattice of $(\tS_{l},<_R)$.
\end{proof}

\begin{proposition}
Fix an affine FPF-involution $z\in\cF_n$. Then $|\cAfpf(z)| = 1$ if and only if $z$ is 321-avoiding in sense that if $i<j<k$ then we do not have $z(i)>z(j)>z(k)$.
\end{proposition}
\begin{proof}
If $|\cAfpf(z)| = 1$, then $\alpha^\fpf_{\min}(z)=\alpha^\fpf_{\max}(z)=[a_1,b_1,a_2,b_2,\cdots,a_l,b_l]^{-1}$, where $b_i<b_j$ for all $i<j$. In this case, if $i<j<k$, 
then two of the elements of $\{i,j,k\}$ belong to $\{a_1,\dots,a_l\}$ or to $\{b_1,\dots,b_l\}$ and it is easy to check that we cannot have $z(i) > z(j) > z(k)$; for example, if $i=a_p<a_q=j$, then we have $z(i)=b_p<b_q=z(j)$.

Conversely, assume $z$ is 321-avoiding and let $\alpha^\fpf_{\min}(z)=[a_1,b_1,a_2,b_2,\cdots,a_l,b_l]^{-1}$. Since $a_i<a_{i+1}<b_{i+1}$,  we do not have $z(a_i)=b_i>z(a_{i+1})=b_{i+1}>z(b_{i+1})=a_{i+1}$. Since $b_{i+1} > a_{i+1}$, we do not have $b_i>b_{i+1}$ for all $i$. Thus $b_i<b_{i+1}$ for all $i$ and $\alpha^\fpf_{\min}(z)=\alpha^\fpf_{\max}(z)$. Therefore $|\cAfpf(z)| = 1$.
\end{proof}

\begin{proposition}\label{transform-thm}
If $\tau$ is any automorphism of $\tS_n$ that preserves $\{s_1,s_2,\dots,s_n\}$, then $\tau$ preserves $\cF_n$ and $\cAfpf(\tau(z))=\{\tau(\pi):\pi\in\cAfpf(z)\}$.
\end{proposition}
\begin{proof}
Since $\tau$ is an automorphism preserving $\{s_1,s_2,\dots,s_n\}$, we have $\tau(\Theta)\in\tpm$
and $\tau$ is length-preserving. Therefore $\tau(w^{-1}\Theta w)=(\tau(w))^{-1}\tau(\Theta)\tau(w)\in\cF_n$ and $\cAfpf(\tau(z))=\{\tau(\pi):\pi\in\cAfpf(z)\}$.
\end{proof}

As a special case, if $\tau$ is the automorphism of $\tS_n$ with $\tau(s_i) = s_{i+1}$ for all $i$, then $\tau$ restricts to a bijection $\cFp \to \cFm$,
and it follows from the preceding result and \cite[Proposition 18]{Lam} that $\Ffpf_z = \Ffpf_{\tau(z)}$.

\begin{proposition}\label{alpha-conj-prop}
For an affine FPF-involution $z$, let $\alpha^\fpf_{\min}(z)=[a_1,b_1,a_2,b_2,\cdots,a_l,b_l]^{-1}$ as in the definition \eqref{amin-def} and assume $s_izs_i\neq z$. Then we have 
$$\alpha^\fpf_{\min}(s_izs_i)=\begin{cases}\alpha^\fpf_{\min}(z)t_{z(i),z(i+1)}&\text{ if }z(i)>i\text{ and }z(i+1)>i+1\\\alpha^\fpf_{\min}(z)s_i&\text{ otherwise}\end{cases}$$
\end{proposition}

\begin{proof}
Write $\alpha^\fpf_{\min}(z)=[a_1,b_1,a_2,b_2,\cdots,a_l,b_l]^{-1}$ as in the definition \eqref{amin-def}.
Let $y=s_izs_i$ and consider the following four cases.
\begin{enumerate}

\item Suppose $a_j=i$ and $a_k\equiv i+1\modu{n})$. Then $\{\overline{(a_j+1,b_j)},\overline{(a_k-1,b_k)}\}\subset\wCyc(y)$. If $a_j+1\le n$ then \[\{(a_j+1,b_j),(a_k-1,b_k)\}=\{(a_k,b_j),(a_j,b_k)\}\subset\Cyc(y)\] while the other cycles remain unchanged. Then we have 
\[\alpha^\fpf_{\min}(y)=[a_1,b_1,\cdots,a_j,b_k,a_k,b_j,\cdots,a_l,b_l]^{-1}=\alpha^\fpf_{\min}(z)t_{b_j,b_k}.\] Therefore $\alpha^\fpf_{\min}(y)=\alpha^\fpf_{\min}(z)t_{z(i),z(i+1)}$.
If $a_j=n$, then 
\[\{(a_j+1-n,b_j-n),(a_k-1+n,b_k+n)\}=\{(1,b_j-n),(n,b_k+n)\}\subset\Cyc(y)\] while the other cycles remain unchanged. Then we have 
\[ \alpha^\fpf_{\min}(y)=[1,b_j-n,\cdots,n,b_k+n]^{-1}
\qquand \alpha^\fpf_{\min}(z)=[1,b_k,\cdots,n,b_j]^{-1}.\] Therefore $\alpha^\fpf_{\min}(y)=\alpha^\fpf_{\min}(z)t_{z(i),z(i+1)}$.

\item Suppose $a_j=i$ and $b_k\equiv i+1\modu{n})$. Then $\{\overline{(a_j+1,b_j)},\overline{(a_k,b_k-1)}\}\subset\wCyc(y)$. If $a_j+1\le n$ and $a_k=a_j$, then \[\{(a_j+1,b_j-1)\}\subset\Cyc(y)\] while the other cycles remain unchanged. Therefore we have \[\alpha^\fpf_{\min}(y)=[a_1,b_1,\cdots,a_j+1,b_j-1,\cdots,a_l,b_l]^{-1}=\alpha^\fpf_{\min}(z)s_i.\]
If $a_j+1\le n$ and $a_k<a_j$, then $\{(a_j+1,b_j),(a_k,b_k-1)\}\subset\Cyc(y)$
 while the other cycles remain unchanged. Therefore we have 
\[\alpha^\fpf_{\min}(y)=[a_1,b_1,\cdots,a_k,b_k-1,\cdots,a_j+1,b_j,\cdots,a_l,b_l]^{-1}=\alpha^\fpf_{\min}(z)s_i.\]
If $a_j=n$, then $\{(a_j+1-n,b_j-n),(a_k,b_k-1)\}=\{(1,b_j-n),(a_k,b_k-1)\}\subset\Cyc(y)$ while the other cycles remain unchanged. Then we have \[
\ba
\alpha^\fpf_{\min}(y)&=[1,b_j-n,a_1,b_1,\cdots,a_k,b_k-1,\cdots,a_{l-1},b_{l-1}]^{-1}
\\&=[a_1,b_1,\cdots,a_k,b_k-1,\cdots,a_{l-1},b_{l-1},n+1,b_j]^{-1}
\ea\]
and
$ \alpha^\fpf_{\min}(z)=[a_1,b_1,\cdots,a_k,b_k,\cdots,a_{l-1},b_{l-1},n,b_j]^{-1}.$ Therefore $\alpha^\fpf_{\min}(y)=\alpha^\fpf_{\min}(z)s_i$.

\item Suppose $b_j\equiv i\modu{n})$ and $a_k\equiv i+1\modu{n})$. Then $\{\overline{(a_j,b_j+1)},\overline{(a_k-1,b_k)}\}\subset\wCyc(y)$. If $a_k-1>0$ and $k=j$, then $\{(a_j-1,b_j+1)\}\subset \Cyc(y)$ while the other cycles remain unchanged. Therefore we have \[\alpha^\fpf_{\min}(y)=[a_1,b_1,\cdots,a_j-1,b_j+1,\cdots,a_l,b_l]^{-1}=\alpha^\fpf_{\min}(z)s_i.\]
If $a_k-1>0$ and $k\neq j$, then $\{(a_j,b_j+1),(a_k-1,b_k)\}\subset \Cyc(y)$ while the other cycles remain unchanged. Then we have \[\alpha^\fpf_{\min}(y)=[a_1,b_1,\cdots,a_j,b_j+1,\cdots,a_k-1,b_k,\cdots,a_l,b_l]^{-1}\] or $[a_1,b_1,\cdots,a_k-1,b_k,\cdots,a_j,b_j+1,\cdots,a_l,b_l]^{-1}$ and \[\alpha^\fpf_{\min}(z)=[a_1,b_1,\cdots,a_j,b_j,\cdots,a_k,b_k,\cdots,a_l,b_l]^{-1}\] or $[a_1,b_1,\cdots,a_k,b_k,\cdots,a_j,b_j,\cdots,a_l,b_l]^{-1}$ respectively. In both cases $\alpha^\fpf_{\min}(y)=\alpha^\fpf_{\min}(z)s_i$.

If $a_k=1$, then $\{(a_j,b_j+1),(a_k-1+n,b_k+n)\}=\{(a_j,b_j+1),(n,b_k+n)\}\subset \Cyc(y)$ while the other cycles remain unchanged. Then we have 
\[
\ba
\alpha^\fpf_{\min}(y)&=[a_2,b_2,\cdots,a_j,b_j+1,\cdots,a_l,b_l,n,b_k+n]^{-1}
=[0,b_k,a_2,b_2,\cdots,a_j,b_j+1,\cdots,a_l,b_l]^{-1}
\ea\] and $\alpha^\fpf_{\min}(z)=[1,b_k,a_2,b_2,\cdots,a_j,b_j,\cdots,a_l,b_l]^{-1}.$ Therefore $\alpha^\fpf_{\min}(y)=\alpha^\fpf_{\min}(z)s_i$.

\item Suppose $b_j\equiv i\modu{n})$ and $b_k\equiv i+1\modu{n})$. Then $\{(a_j,b_j+1),(a_k,b_k-1)\}\subset \Cyc(y)$ while the other cycles remain unchanged. Then we have \[\alpha^\fpf_{\min}(y)=[a_1,b_1,\cdots,a_j,b_j+1,\cdots,a_k,b_k-1,\cdots,a_l,b_l]^{-1}\] or $[a_1,b_1,\cdots,a_k,b_k-1,\cdots,a_j,b_j+1,\cdots,a_l,b_l]^{-1}$ and \[\alpha^\fpf_{\min}(z)=[a_1,b_1,\cdots,a_j,b_j,\cdots,a_k,b_k,\cdots,a_l,b_l]^{-1}\] or $[a_1,b_1,\cdots,a_k,b_k,\cdots,a_j,b_j,\cdots,a_l,b_l]^{-1}$ respectively. In both cases $\alpha^\fpf_{\min}(y)=\alpha^\fpf_{\min}(z)s_i$.
\end{enumerate}
\end{proof}

\begin{proposition}\label{length-t-conj-prop}
Fix $i<j\not\equiv i\modu{n})$, $t=t_{ij}$, and $z\in\cF_n$.
For $w \in \tS_n$, define 
$$E(w,i,j)=\{k \in \ZZ : i<k<j\text{ and }k\not\equiv i\modu{n})\text{ and }w(k)\text{ is between }w(i)\text{ and }w(j)\}$$
and let $\delta(w,i,j) = |E(w,i,j)|$.
 If $z(i)\not\equiv j\modu{n})$, then we have 
$$\ell(tzt)=\begin{cases}\ell(z)+2\delta(tz,i,j)+2\delta(z,i,j)+2&\text{ if }z(i)<z(j)\\\ell(z)-2\delta(tzt,i,j)-2\delta(zt,i,j)-2&\text{ if }z(i)>z(j).\end{cases}$$
If $z(i)=j+mn$ for $m\in\ZZ$, so that $z(j) = i-mn$, then we have
$$\ell(tzt)=\begin{cases}\ell(z)+2\delta(tz,i,j)+2\delta(z,i,j)+2&\text{ if }m<-\frac{j-i}{2n}\\
\ell(z)+2\delta(tz,i,j)-2\delta(zt,i,j)>\ell(z)&\text{ if }-\frac{j-i}{2n}<m<0\\
\ell(z)+2\delta(tz,i,j)-2\delta(zt,i,j)=\ell(z)&\text{ if }m=0\\
\ell(z)+2\delta(tz,i,j)-2\delta(zt,i,j)<\ell(z)&\text{ if }0<m<\frac{j-i}{2n}\\
\ell(z)-2\delta(tzt,i,j)-2\delta(zt,i,j)-2&\text{ if }m>\frac{j-i}{2n}.\end{cases}$$
\end{proposition}
\begin{proof}
By Lemma~\ref{t-length-lem} we have 
$$\ell(tzt)=\begin{cases}\ell(tz)+2\delta(tz,i,j)+1&\text{ if }tz(i)<tz(j)\\\ell(tz)-2\delta(tzt,i,j)-1&\text{ if }tz(i)>tz(j)\end{cases}
\quand
\ell(tz)=\begin{cases}\ell(z)+2\delta(z,i,j)+1&\text{ if }z(i)<z(j)\\\ell(z)-2\delta(zt,i,j)-1&\text{ if }z(i)>z(j).\end{cases}$$
If $z(i)\not\equiv j\modu{n})$, then we have $tz(i)=z(i)$ and $tz(j)=z(j)$,
while if $z(i)=j+mn$ for some $m \in \ZZ$, then $tz(i)=i+mn$, $tz(j)=j-mn$, $z(i)=j+mn$, and $z(j)=i-mn$. In each case it is straightforward to
derive the desired length formulas from the preceding equations.
%

Assume $z(i) = j+mn$ for some $m \in \ZZ$ with $-\frac{j-i}{2n} < m < \frac{j-i}{2n}$.
It remains to show that $\ell(z) - \ell(tzt)$ has the same sign as $m$.
We can write $z=yt_{i,j+mn}$ where $y \in \tI_n$ fixes all elements of $\{i,j\}+n\ZZ$. By Lemma~\ref{t-length-lem} we have $\ell(z)=\ell(y)+2\delta(y,i,j+mn)+1$. Also, $tzt=yt_{i+mn,j}$ and $\ell(tzt)=\ell(y)+2\delta(y,i+mn,j)+1$. When $-\frac{j-i}{2n}<m<\frac{j-i}{2n}$, if $m\neq0$, we have $i<i+mn<j<j+mn$ or $i+mn<i<j+mn<j$. In the first case, $m>0$ and $E(y,i+mn,j)\subset E(y,i,j+mn)$. However, $j+n\in E(y,i,j+mn)$ but $j+n\not\in E(y,i+mn,j)$, so $\delta(y,i,j+mn)>\delta(y,i+mn,j)$ and $\ell(tzt)<\ell(z)$. In the second case, $m<0$ and $E(y,i,j+mn)\subset E(y,i+mn,j)$. However, $j-n\in E(y,i+mn,j)$ but $j-n\not\in E(y,i,j+mn)$, so $\delta(y,i+mn,j)>\delta(y,i,j+mn)$ and $\ell(tzt)>\ell(z)$.

 Finally, if $m=0$ and $z(i)=j$, we have $tzt=z$ and $\ell(tzt)=\ell(z)$.
\end{proof}

\begin{corollary}\label{length-s-conj-prop}
For $z \in \cF_n$ and $i \in \ZZ$, we have 
\[\ellfpf(s_izs_i)=\begin{cases}
\ellfpf(z)-1&\text{ if }z(i)>z(i+1)\neq i\\
\ellfpf(z)&\text{ if }z(i)=i+1\\
\ellfpf(z)+1&\text{ if }z(i)<z(i+1).
\end{cases}\]
\end{corollary}
\begin{proof}
This is a special case of Proposition~\ref{length-t-conj-prop}, since $\ellfpf(z)=\frac{1}{2}\ell(z)-\frac{n}{4}$.
\end{proof}

\begin{corollary}\label{length-t-cor}
Fix $i<j\not\equiv i\modu{n})$, $t=t_{ij}$, and $z\in\cF_n$.
If $\ell(tzt)=\ell(z)+2$, then $\ell(zt)=\ell(tz)=\ell(z)+1$.
\end{corollary}
\begin{proof}
Assume $\ell(tzt)=\ell(z)+2$; then we have two cases. In the first case, if $z(i)\not\equiv j\modu{n})$, 
then we must have $z(i)<z(j)$ and $\delta(tz,i,j)=\delta(z,i,j)=0$, so the result is clear from Lemma~\ref{t-length-lem}.
For the second case, suppose $z(i)=j+mn$ and $m<0$. 
If $m<-\frac{j-i}{2n}$, then the situation is similar to the first case. Suppose $-\frac{j-i}{2n}<m<0$.
Following the proof of Proposition~\ref{length-t-conj-prop}, we argue that this case leads to a contradiction and therefore cannot occur.
 We have $j-n\in E(y,i+mn,j)$ but $j-n\not\in E(y,i,j+mn)$. Also there exists $j'\equiv j\modu{n})$ such that $i+mn<j'<i$. Then $j'\in E(y,i+mn,j)$ but $j'\not\in E(y,i,j+mn)$. So we have $\delta(y,i+mn,j)-\delta(y,i,j+mn)\ge2$ and $\ell(tzt)-\ell(z)\ge4$, which is impossible since $\ell(tzt)=\ell(z)+2$.
\end{proof}

Now we can introduce the analogues of codes and descents for elements of $\cF_n$.

\begin{definition}
The \emph{FPF-involution code} of $z \in \cF_n$ is the sequence $\cfpf(z)=(c_1,c_2,\dots,c_n)$
where $c_i$ is the number of integers $j \in \ZZ$ with $i<j$ and $z(i)>z(j)$ and $i>z(j)$.
\end{definition}

\begin{definition}
An integer $i \in \ZZ$ is an \emph{FPF-visible descent} of $z \in \cF_n$ if $\min\{i,z(i)\}>z(i+1)$.
Let \[\DesF(z)=\{s_i:i\in\ZZ\text{ is an FPF-visible descent of }z\}.\]
\end{definition}

\begin{lemma}\label{desF-lem}
For $z \in \cF_n$, we have $\DesF(z) =\DesR(\alpha^\fpf_{\min}(z))$ and $\cfpf(z) = c(\alpha^\fpf_{\min}(z))$.
\end{lemma}
\begin{proof}
Fix $z \in \cF_n$ and integers $i<j$. 
It is clear from the definition that $\alpha^\fpf_{\min}(z)(i) > \alpha^\fpf_{\min}(z)(j)$
if and only if either
$z(j)<i<z(i)<j$
or
$z(j)<i<j<z(i)$
or
$z(j)<z(i)<i<j$.
One of these cases occurs precisely when $z(i)>z(j)$ and $i>z(j)$.
We conclude that $\cfpf(z) = c(\alpha^\fpf_{\min}(z))$ and, taking $j=i+1$, that $\DesF(z) =\DesR(\alpha^\fpf_{\min}(z))$.
\end{proof}

\begin{corollary}\label{fcode-cor}
Suppose $z \in \cF_n$ and $\cfpf(z) = (c_1,c_2,\dots,c_n)$.
Then $\ellfpf(z)=c_1+c_2+\dots+c_n$, and an integer
$i\in\ZZ$ is an FPF-visible descent of $z$ 
if and only if $c_i>c_{i+1}$, interpreting indices modulo $n$.
\end{corollary}
\begin{proof}
This is immediate from Theorem~\ref{length-alphamin-thm} and Lemma~\ref{desF-lem} given the discussion in Section~\ref{pre-sect}.
\end{proof}

\begin{corollary}
For $z \in \cF_n$, the following are equivalent:
\ben
\item[(a)] $\DesF(z) \subset \{ s_n\}$.
\item[(b)] $\cfpf(z)$ is weakly increasing.
\item[(c)] $\alpha^\fpf_{\min}(z)^{-1}$ is Grassmannian. 
\een
\end{corollary}
\begin{proof}
Parts (a) and (b) are equivalent because of \ref{fcode-cor}. The element $\alpha^\fpf_{\min}(z)^{-1}$ is Grassmannian if and only if $\alpha^\fpf_{\min}(z)(1)<\alpha^\fpf_{\min}(z)(2)<\cdots<\alpha^\fpf_{\min}(n)$, which is equivalent to $\DesF(z)=\DesR(\alpha^\fpf_{\min}(z))\subset \{ s_n\}$.
\end{proof}

Although this result makes it seem that the functions $\Ffpf_z$ for $z \in \cF_n$ with $\DesF(z) \subset \{s_n\}$ would be a good analogue of affine Schur functions, these power series do not even span $\ZZ\spanning\{ \Ffpf_y : y \in \cF_n\}$.

\begin{proposition}\label{fcode-prop}
If $i \in \ZZ$ is an FPF-visible descent of $z$ then 
$\cfpf(s_i  z  s_i) = (c_1,\dots,c_{i+1},c_i - 1,\dots, c_n)$
 interpreting indices modulo $n$.
\end{proposition}

\begin{proof}
Let $i$ be an FPF-visible descent of $z$, so that $i$ is a descent of $\alpha^\fpf_{\min}(z)$ by Lemma~\ref{desF-lem}.
By Proposition~\ref{alpha-conj-prop}, we have $\alpha^\fpf_{\min}(z) s_i =\alpha^\fpf_{\min}(s_i  z  s_i)$.
The result therefore follows from \eqref{ccc-eq} and Lemma~\ref{desF-lem}.
\end{proof}

\begin{corollary} 
The maps
$\cfpf : \cFp \to \NN^n - \PP^n$ and $\cfpf : \cFm \to \NN^n - \PP^n$ are injective.
\end{corollary}

\begin{proof}
This follows by induction from Corollary~\ref{fcode-cor} and Proposition~\ref{fcode-prop}.
\end{proof}

\begin{definition}
The \emph{FPF-involution shape} $\nu(z)$ of $z \in \cF_n$ is the transpose of the partition sorting $\cfpf(z)$.
\end{definition}

Recall the partitions $\lambda(\pi)$ and $\lambda'(\pi)=\lambda(\pi^*)$ given in Section~\ref{pre-sect}.

\begin{lemma}\label{lambda-nu-lem}
If $z \in \cF_n$, then $\nu(z)=\lambda(\alpha^\fpf_{\max}(z))=\lambda(\alpha^\fpf_{\min}(z)^{-1})$. Moreover, $\nu(z)^*=\lambda'(\alpha^\fpf_{\min}(z)).$
\end{lemma}

\begin{proof}
If $z \in \cF_n$, then $\nu(z)=\lambda(\alpha^\fpf_{\min}(z)^{-1})$ holds by Lemma~\ref{desF-lem}, so we just need to show 
that $\lambda(\alpha^\fpf_{\max}(z))=\lambda(\alpha^\fpf_{\min}(z)^{-1})$, or equivalently that sorting $c(\alpha^\fpf_{\max}(z)^{-1})$  gives same thing as sorting $c(\alpha^\fpf_{\min}(z))$.

Let $l=\frac{n}{2}$, $\pi=\alpha^\fpf_{\max}(z)^{-1}$, and $\tau=\alpha^\fpf_{\min}(z)$. 
We assume $z \in \cFp$ and set $\Theta = \Theta^+$; the argument for the case when $z \in \cFm$ is similar.
Using \eqref{ccc-eq}, we compute that 
\[c_{2i}(\pi)=c_{2i-1}(\pi\Theta)\text{ and } c_{2i-1}(\pi)=c_{2i}(\pi\Theta)+1\quad\text{while}\quad
c_{a_i}(\tau)=c_{a_i}(\Theta\tau)+1\text{ and } c_{b_i}(\tau)=c_{b_i}(\Theta\tau)\] 
where $a_i < b_i = z(a_i)$ are as in \eqref{amin-def}. Let $f_1<f_2<\dots<f_l$ be the numbers $f \in [n]$ with $z(f) < f$ and set $e_i=z(f_i)$.
Then $c_{2i-1}(\pi)=\#\{j>i:e_i<e_j\}=\#\{j:e_j<e_i<f_i<f_j\}$ and $c_{b_i}(\tau)=\#\{j<i:b_i<b_j\}=\#\{j:a_j<a_i<b_i<b_j\}$, which means that $\{c_{2i-1}(\pi)\}_{i=1}^{l}=\{c_{b_i}(\tau)\}_{i=1}^{l}$.  
Thus, it suffices to show that 
sorting $c(\pi\Theta)$ gives same thing as sorting $c(\tau \Theta)$. This claim is equivalent to \cite[Lemma 3.19]{MZ}.
%
%
%
%
\end{proof}

\begin{theorem}\label{i<-thm}
Suppose $n$ is positive even integer and $z \in \cF_n$. The following properties then hold:
\ben
\item[(a)] $\Ffpf_z \in m_{\nu(z)} + \sum_{\mu < \nu(z)} \NN m_\nu \subset \Sym^{(n)}.$
\item[(b)] $ \Ffpf_z \in F_{\nu(z)} + \sum_{\mu < \nu(z)} \NN F_\nu.$
\item[(c)] $ \Ffpf_z \in  F_{\nu(z)^*} + \sum_{ \nu(z)^* <^* \mu} \NN F_\nu.$
\een
Here, the symbol $<$ denotes the dominance order on partitions and $\lambda <^* \mu$ if and only if $\mu^* < \lambda^*$.
\end{theorem}
\begin{proof}
Given Lemmas~\ref{lambda-nu-lem} and \cite[Lemma 3.20]{MZ},
the result is immediate from Theorems~\ref{uni-thm} and \ref{schur-thm}.
\end{proof}

\begin{example}
Again let $z=t_{1,6}t_{3,8}=[6,-3,8,-1]\in\cF_4$, so that we have $\alpha^\fpf_{\min}(z)=[3,0,5,2]=\alpha^\fpf_{\max}(z)=[1,6,3,8]^{-1}$. Then $\cfpf(z)=c(\alpha^\fpf_{\min}(z))=(2,0,2,0)$ and $c(\alpha^\fpf_{\max}(z)^{-1})=(0,2,0,2)$ so
\[
\nu(z) = \lambda(\alpha_{\max}(z))=\lambda(\alpha_{\min}(z)^{-1})=(2,2).
\]
\end{example}

The affine FPF involution code and shape are analogues of the affine involution code and shape in \cite{MZ}. As noted there, the images of the code and shape functions are still unknown but they have an interesting image when restricted to the finite subgroup $S_n \subset\tS_n$. Analogously, it is still an open problem to characterize the images of the maps $\cfpf : \cF_n \to \NN^n - \PP^n$ and $\nu : \cF_n \to \Par^n$.


\begin{problem}
What is the image of $\nu : \cF_n \to \Par^n$?
\end{problem}


Our definition of the functions $\Ffpf_z$ is given in terms of a
  right-handed version of FPF-atoms, although there exists an equally natural left-handed version.
Interesting, this convention seems to be irrelevant:

\begin{conjecture}\label{inv-cong-conj}
If $z \in \cF_n$ then $\Ffpf_z = \sum_{\pi \in \cAfpf(z)} \tF_{\pi^{-1}} \in \Sym^{(n)}.$
\end{conjecture}
By Theorem~\ref{inv-cong-thm}, this conjecture is equivalent to the identity $\omega^+(\Ffpf_z) = \Ffpf_z$.

In \cite{MZ}, we studied another family of symmetric functions $\iF_z$ indexed by all involutions $z \in \tI_n$ in the affine symmetric group, generalizing constructions of Hamaker, Marberg, and Pawlowski from \cite{HMP1,HMP4}.
Let 
\[\Omega_n= \ZZ\spanning\left\{\iF_z : z \in \tI_n \cap S_n\right\} 
\qquand
\tilde\Omega_n = \ZZ\spanning\left\{\iF_z : z \in \tI_n \right\}.\]
For each even positive integer $n$, similarly define
\[\Omega^\fpf_n= \ZZ\spanning\left\{\Ffpf_z : z \in \cF_n \cap S_n\right\} 
\qquand
\tilde\Omega^\fpf_n = \ZZ\spanning\left\{\Ffpf_z : z \in \cF_n \right\}.\]

It follows from results in \cite{HMP4} that $\Omega_n$ is exactly the $\ZZ$-span of the Schur $P$-functions $P_\lambda$ where $\lambda$ ranges over all strict partitions contained in the partition $(n-1,n-3,n-5,\ldots)$.
Also, each $P_\lambda$ in this basis is equal to $\iF_z$ for some $z=z^{-1}$ in $S_n$. Moreover, every $\iF_z$ for $z=z^{-1}$ in $S_n$ is a nonnegative integer combination of the Schur $P$-function basis.
It follows from results in \cite{HMP5} that if $n$ is even then $\Omega^\fpf_n=\Omega_{n-1}$,
and that every $\Ffpf_z$ is also a nonnegative integer combination of Schur $P$-functions.

The map $\omega^+$ restricts on $\Omega_n$ and $\Omega^\fpf_n$ to the better-known involution $\omega$ that sends $s_\lambda$ to $s_{\lambda^T}$ for all partitions $\lambda$. This map fixes every Schur $P$-function and therefore also every $\iF_z$ and $\Ffpf_z$ for $z$ in the finite group $S_n$. So Conjecture~\ref{inv-cong-conj} is at least true if the index $z$ is in $S_n \subsetneq \tS_n$.

There are many open problems related to the spaces $\tilde\Omega_n$ and $\tilde\Omega^\fpf_n$. First, in general, $\Ffpf_z$ for $z \in \cF_n$ is not a linear combination of Schur $P$-functions. Second, at least for $n=6,8$, and probably for all $n\ge 6$, one can check by computer that no subset of $\{\Ffpf_z:z\in\cF_n\}$ gives a basis for $\tilde\Omega_n^\fpf$ that is ``positive" in sense that each $\Ffpf_z$ is positive integer combination of basis elements. Thus, it is an open problem to characterize $\tilde{\Omega}_n$ and $\tilde{\Omega}^\fpf_n$ and to find good bases for these submodules.

Assume $n$ is a positive integer. Computer calculations suggest the following conjectures:

\begin{conjecture}\label{omega1-conj}
If $n$ is even then $\tilde{\Omega}_n=\tilde{\Omega}^\fpf_n$.
\end{conjecture}

This is mysterious, because  $\left\{\iF_z:z\in\tI_n\right\}\neq\left\{\Ffpf_z:z\in\cF_n\right\}$ and $\cF_n \subsetneq \tI_n$.

\begin{conjecture}\label{omega3-conj}
If $n$ is even then $\tilde{\Omega}_{n-1}\subsetneq \tilde{\Omega}_n$ and $\tilde{\Omega}_{n-1}\not\subset \tilde{\Omega}_{n+2}$.
\end{conjecture}
Here for $n=2$, we define $\tilde{\Omega}_{1}=\ZZ$. Finally:
\begin{conjecture}
If $n$ is even then $\tilde{\Omega}^\fpf_n\not\subset \tilde{\Omega}^\fpf_{n+2}$.
\end{conjecture}
This conjecture would be a corollary of Conjectures~\ref{omega1-conj} and \ref{omega3-conj}.

One nice consequence of these conjectures is that although we might think that FPF involutions are a special case of ordinary involutions, in fact the FPF case is equally general at least at the level of Stanley symmetric functions. This is counterintuitive and worth more consideration.

\section{Quasiparabolic sets}\label{qp-sect}

Let $(W,S)$ be an arbitrary Coxeter system with length function $\ell$. For the general theory of Coxeter systems, see \cite{H}. In \cite{RV}, Rains and Vazirani give a definition of a quasiparabolic $W$-set, which we now review. Let 
$R(W) = \{ wsw^{-1} : w\in W,\ s \in S\}$ denote the set of reflections in $W$. 

\begin{definition}[{Rains and Vazirani \cite[Definition 2.1]{RV}}]
A \emph{scaled $W$-set} is a pair $(X,\h)$ with $X$ a $W$-set and $\h$: $X \to \ZZ$ a function such that $|\h(sx)-\h(x)|\le1$ for all $s\in S$. An element $x\in X$ is \emph{$W$-minimal} if $\h(sx)\ge  \h(x)$ for all $s\in S$.
\end{definition}

\begin{definition}[{Rains and Vazirani \cite[Definition 2.3]{RV}}]\label{qp-def}
A \emph{quasiparabolic $W$-set} is a scaled $W$-set $(X,\h)$ satisfying the following two properties:
\begin{enumerate}
    \item[$\bullet$] (QP1) For all $r\in R(W)$, $x\in X$, if $\h(rx)=\h(x)$, then $rx = x$.
    \item[$\bullet$] (QP2) For all $r\in R(W)$, $x\in X$, $s\in S$, if $\h(rx) > \h(x)$ and $\h(srx) < \h(sx)$, then $rx = sx$.
\end{enumerate}
\end{definition}

Before we stating more results, we mention a fundamental example.

\begin{example}\label{qp-ex1}
Choose a subset $J\subset S$ and let $W_J = \langle J\rangle$ be a standard parabolic subgroup. The set of left cosets $X = W/W_J$ is a quasiparabolic $W$-set with respect to the height function $\h : X \to \ZZ$ defined by $\h(C) = \min_{w \in C} \ell(w)$ for a left coset $C$ of $W_J$ in $W$.
\end{example}

The following technical lemma is useful for verifying condition (QP2).

\begin{lemma}[{Rains and Vazirani \cite[Lemma 2.4]{RV}}]\label{ht+1-lem}
Let $(X,\h)$ be a scaled $W$-set, and suppose $r\in R(W)$, $s\in S$, $x\in X$ are such that $\h(rx)>\h(x)$ and $\h(srx)<\h(sx)$. Then $\h(rx)=\h(sx)=\h(x)+1=\h(srx)+1$.
\end{lemma}

\begin{definition}[{Rains and Vazirani \cite[Definition 2.11]{RV}}]\label{r-e-def}
Let $(X,\h)$ be a quasiparabolic $W$-set with $W$-minimal element $x_0$. Suppose $x\in Wx_0\subset X$ is in the $W$-orbit of $x_0$ and has height $\h(x)=k+\h(x_0)$. Then we call $s_1s_2\cdots s_kx_0$ a \emph{reduced expression} for $x$ if $x=s_1s_2\cdots s_kx_0$, $s_i\in S$. By abuse of notation, we also call $wx_0$ reduced where $w =s_1s_2\cdots s_k$ is the corresponding reduced expression.
\end{definition}

The product $W\times W$ is itself a Coxeter group with respect to the set of 
simple generators $\{1\} \times S \sqcup S \times \{1\}$.

\begin{theorem}[{Rains and Vazirani \cite[Theorem 3.1]{RV}}]\label{W-qp-lem}
The set $W$, together with the height function $\h=\ell$ and the $W\times W$-action
$(w',w'')\cdot w :=w'w(w'')^{-1}$ for $w,w',w'' \in W$
is a quasiparabolic $W\times W$-set.
\end{theorem}

The \emph{Bruhat order} on $W$ is the transitive closure of the relation with $u<v$
if $\ell(v) =\ell(u)+1$ and $v = ru$ for some $r \in R(W)$.
Similarly, there is a Bruhat order on each quasiparabolic set.

\begin{definition}[{Rains and Vazirani \cite[Definition 5.1]{RV}}]\label{Bruhat-defn}
Let $(X, \h)$ be a quasiparabolic $W$-set. The \emph{Bruhat order} on $X$ is the transitive closure of the relation that for $x\in X$, $r\in R(W)$, $x\le rx$ if $\h(x)\le \h(rx)$.
\end{definition}

If $X=W/W_J$ as in Example~\ref{qp-ex1}, then this ``quasiparabolic'' Bruhat order coincides with the usual Bruhat order on $W$ for elements in the same bounded $W_J$-orbit  \cite[Theorem 5.12]{RV}. In particular, the Bruhat order of $W$ viewed as a $W\times W$-set is the same as the usual Bruhat order.

We now specialize to the case when $W=\tS_n$
where $n$ is even and $S = \{s_1,s_2,\dots,s_n\}$.
Our main result is to show that the set of fixed-point-free involutions $\cF_n\subset \tS_n$
is naturally a quasiparabolic $\tS_n$-set. This is interesting on account of the following.

In \cite[Theorem 4.6]{RV}, Rains and Vazirani showed that the conjugacy class of fixed-point-free involutions in the finite symmetric group $S_n$  is a quasiparabolic $S_n$-set. This is an important example of a quasiparabolic set that is not parabolic. 
An involution $z =z^{-1} \in W$ is \emph{perfect} if for all reflections $r \in R(W)$, it holds that $r$ commutes with $rzr$.
Rains and Vazirani proved \cite[Theorem 4.6]{RV} that all $W$-conjugacy classes of perfect involutions are quasiparabolic sets.

While it is easy to show that the fixed-point-free involutions in $S_n$ are perfect, this does not hold for elements of $\cF_n$. (For a counterexample, take $z=\tp\in\cF_4$ and $r=t_{4,7}$.) Therefore, we need some new techniques to show that $\cF_n$ is quasiparabolic.

In the following theorem, $\tS_n$ acts on $\cF_n$ by conjugation.

\begin{theorem}\label{quasiparabolic set}
The pair $(\cF_n,\ellfpf)$ is a quasiparabolic $\tS_n$-set.
\end{theorem}

There are two $\tS_n$-minimal elements in $\cF_n$, given by $\tp$ and $\tm$.

\begin{proof}
Fix $z \in \cF_n$ and $i<j\not\equiv i\modu n)$ and let $t=t_{ij} \in \tS_n$. According to Proposition~\ref{length-s-conj-prop}, we see that $(\cF_n,\ellfpf)$ is a scaled $\tS_n$-set. The first condition in Definition~\ref{qp-def} is equivalent to the claim that $\ellfpf(tzt) = \ellfpf(z)$ only if $tzt=z$. This follows from Proposition~\ref{length-t-conj-prop}.

Now fix $k \in \ZZ$, let $s=s_k \in \tS_n$, and assume that $\ellfpf(tzt)>\ellfpf(z)$ and  $\ellfpf(stzts) < \ellfpf(szs)$. The second condition in Definition~\ref{qp-def} is equivalent to the claim that $szs = tzt$. Lemma~\ref{ht+1-lem} implies that 
$$\ellfpf(tzt)=\ellfpf(szs)=\ellfpf(z)+1=\ellfpf(stzts)+1.$$
In other words, since $\ellfpf(z)=\frac{1}{2}\ell(z)-\frac{n}{4}$, we have
$$\ell(tzt)=\ell(szs)=\ell(z)+2=\ell(stzts)+2.$$
Thus, by Corollary~\ref{length-t-cor}, we have $\ell(zt)=\ell(tz)=\ell(z)+1$ and $\ell(tzts)=\ell(stzt)=\ell(tzt)-1=\ell(z)+1$.

Now consider the element $stz$, which has length $\ell(stz)\in \{\ell(tz)-1, \ell(tz)+1\}=\{\ell(z), \ell(z)+2\}$. We have two cases. Recall that $(x,y)\cdot w:= xwy^{-1}$ for $x,y,w \in\tS_n$.
If $\ell(stz)=\ell(z)$, then by the fact that $\ell(sz)=\ell(tz)=\ell(z)+1$, we have $\ell((t,1)\cdot z)>\ell(z)$ and $\ell((s,1)\cdot(t,1)\cdot z)<\ell((s,1)\cdot z)$. According to Lemma~\ref{W-qp-lem}, we conclude that $sz=tz$. But $\tS_n$ is a group, so we can cancel $z$ to find $s=t$, and thus in particular $szs=tzt$ as required.

If $\ell(stz)=\ell(z)+2$, then on the one hand
$\ell(stz)=\ell(tzt)=\ell(tz)+1=\ell(stzt)+1,$
which means
$$\ell((1,t)\cdot tz)>\ell(tz)\text{ and }\ell((s,1)\cdot(1,t)\cdot tz)<\ell((s,1)\cdot tz),$$
so by Lemma~\ref{W-qp-lem} we have $stz=tzt$, while on the other hand
$$\ell(stz)=\ell(szs)=\ell(sz)+1=\ell((sts)szs)+1,$$
which means
$$\ell((sts,1)\cdot sz)>\ell(sz)\text{ and }\ell((1,s)\cdot(sts,1)\cdot sz)<\ell((1,s)\cdot sz),$$
so by Lemma~\ref{W-qp-lem} we have $stz=(sts)sz=szs$. Therefore $tzt=szs$ as required.
\end{proof}

Since $\cF_n$ is a quasiparabolic $\tS_n$-set, it inherits a \emph{Bruhat order} from Definition~\ref{Bruhat-defn}.
We use the symbol $\leq_F$ to denote this partial order.
We write $x\lessdot_F y$ if $y \in \cF_n$ is a \emph{Bruhat cover} of $x \in \cF_n$, in other words, if there exists a transposition $t \in \tS_n$ such that $y=txt$ and $\ellfpf(y) = \ellfpf(x) + 1$.
Let $\le$ denote the usual Bruhat order on the Coxeter group $\tS_n$ and let $\lessdot$ denote the covering relation in $\leq$.

\begin{proposition}
Suppose $y,z \in \cF_n$. If $y\le_Fz $ then $y\le z$.
\end{proposition}

\begin{proof}
We have $y \lessdot_F z$ if and only if $z=tyt$ where $t =t_{ij}$ and $\ellfpf(tyt) = \ellfpf(y)+1$. 
But by Corollary~\ref{length-t-cor}, $\ellfpf(tyt) = \ellfpf(y)+1$ only if $\ell(yt) = \ell(y)+1$ and $\ell(tyt) = \ell(ty) +1$,
which implies $y < yt < tyt=z$. 
%
%
\end{proof}

We believe that the converse of this result is also true. 
Rains and Vazirani  showed in \cite{RV} that this at least holds when the orders are restricted to $\cF_n \cap S_n$, using the properties of perfect involutions.

\begin{corollary}\label{atom-cover-cor}
Suppose $y,z \in \cF_n$ are such that $\sfpf(y) = \sfpf(z)$ and $w \in \cAfpf(z)$. Then $y\lessdot_F z$ if and only if there exists $v \in \cAfpf(y)$ with $v \lessdot w$.
\end{corollary}
\begin{proof}
Since $(\cF_n, \ellfpf)$ is quasiparabolic set, this result follows from \cite[Theorem 5.15]{RV}.
\end{proof}

We discuss some other applications of Theorem~\ref{quasiparabolic set}.
Let $\cA=\ZZ[v,v^{-1}]$ and let $(W,S)$ be a Coxeter system with length function $\ell$. The \emph{Iwahori-Hecke algebra} of
$(W,S)$ is the $\cA$-algebra
$\cH=\cH(W,S)$, with a basis given by the symbols $H_w$ for $w\in W$, whose multiplication is uniquely determined by the condition that
$$H_sH_w=\begin{cases}
H_{sw}&\ell(sw)>\ell(w)\\
H_{sw}+(v-v^{-1})H_w&\ell(sw)<\ell(w)\end{cases}\qquad \text{for }s\in S\text{ and }w\in W.$$
%
The following definitions are from Stembridge's papers \cite{Stem1,Stem2}.
\begin{definition}
An \emph{$I$-labeled graph} for a finite set $I$ is a triple $\Gamma=(V,\omega,\tau)$ where
\ben
\item[(i)] $V$ is a finite vertex set;
\item[(ii)] $\omega: V\times V\to \cA$ is a map;
\item[(iii)] $\tau: V\to \mathcal{P}(I)$ is a map assigning a subset of $I$ to each vertex.
\een
We view $\Gamma$ as a weighted directed graph on the vertex set $V$ with an edge  $x\xrightarrow{\omega(x,y)} y$ if $\omega(x,y)\neq 0$.
\end{definition}

\begin{definition}
An $S$-labeled graph $\Gamma=(V,\omega,\tau)$ if a \emph{$W$-graph} if the free $\cA$-module generated by $V$ can be given an $\cH$-module structure with
\[
H_sx=\begin{cases}vx&s\not\in\tau(x)\\
-v^{-1}x+\sum_{y\in V;s\not\in\tau(y)}\omega(x,y)y&s\in\tau(x)\end{cases}
\qquad \text{ for }s\in S\text{ and }x\in V.
\]
\end{definition}

\begin{example}
There exist a unique ring homomorphism $\cH\to\cH$ with $v\mapsto v^{-1}$ and $H_s\mapsto H_s^{-1}$; we denote this map by $H\mapsto\overline{H}$, and refer to it as the \emph{bar operator} of $\mathcal{H}$.
Write $<$ for the Bruhat order on $W$.
By well-known results of Kazhdan and Lusztig \cite{KL},
for each $w\in W$ there is a unique $\underline H_w\in\mathcal{H}$ with
$$\underline{H}_w = \overline{\underline{H}}_w\in H_w+\sum_{y<w}v^{-1}\ZZ[v^{-1}]H_y.$$
The elements $\{\underline{H}_w\}_{w\in W}$ form an $\cA$-basis for $\mathcal{H}$, called the \emph{Kazhdan-Lusztig basis}.
Define $h_{x,y}\in\ZZ[v^{-1}]$ for $x,y \in W$ such that $\underline{H}_y=\sum_{x\in W}h_{x,y}H_y$,
and let $\mu(x,y)$ be the coefficient of $v^{-1}$ in $h_{x,y}$. 
Finally, let 
$$ \tau(x) = \{ s \in S : \ell(sx) < \ell(x)\}\qquand \omega(x,y)=\begin{cases}\mu(x,y)&\text{if }\tau(x)\not\subset\tau(y)\\0,&\text{otherwise}\end{cases}$$
for $x,y \in W$.
Then $\Gamma = (W,\omega,\tau)$ is a $W$-graph \cite{KL}.
\end{example}

%
%

We turn back to the case $W=\tS_n$. Let $\cM =\cA\text{-span}\{M_z:z\in\cF_n\}$ and $
\cN=\cA\text{-span}\{N_z:z\in\cF_n\}$
denote the free $\cA$-modules with bases given by the symbols $M_z$ and $N_z$ for $z\in\cF_n$. We call $\{M_z\}_{z\in\cF_n}$ and $\{N_z\}_{z\in\cF_n}$ the standard bases of $\cM$ and $\cN$, respectively. 
Let $S = \{s_1,s_2,\dots,s_n\}$.

\begin{corollary}
Both $\cM$ and $\cN$ have unique $\cH$-module structures such that if $s\in  S$ and $z\in\cF_n$ then
{\small\[H_sM_z=\begin{cases}
M_{szs}&\ellfpf(szs)>\ellfpf(z)\\
vM_{z}&\ellfpf(szs)=\ellfpf(z)\\
M_{szs}+(v-v^{-1})M_z&\ellfpf(szs)<\ellfpf(z)\end{cases}
\quand
H_sN_z=\begin{cases}
N_{szs}&\ellfpf(szs)>\ellfpf(z)\\
-v^{-1}N_{z}&\ellfpf(szs)=\ellfpf(z)\\
N_{szs}+(v-v^{-1})N_z&\ellfpf(szs)<\ellfpf(z).\end{cases}
\]}
\end{corollary}
\begin{proof}
The result is a special case of \cite[Theorem 7.1]{RV} since $(\cF_n,\ellfpf)$ is a quasiparabolic $\tS_n$-set.
\end{proof}

Rains and Vazirani's general theory of quasiparabolic sets gives us for free the $\mathcal{H}$-module structures described in the previous result. These modules are potentially interesting to study on their own. We note a few special properties which follow from results in \cite{Mar1}.

We write $f \mapsto \overline{f}$ for the automorphism of $\cA$ interchanging $v$ and $v^{-1}$. A map $U\to V$ of $\cA$-modules is \emph{$\cA$-antilinear} if $x\mapsto y$ implies $ax\mapsto\overline{a}y$ for all $a\in\cA$.

\begin{corollary}
The $\cH$-modules $\cM$ and $\cN$ have the following properties:
\ben
\item[(a)] There are unique $\cA$-antilinear maps $\cM\to\cM$ and $\cN\to\cN$, which we write as $X \mapsto\overline X$,
with 
$$\overline{HM}=\overline{H}\cdot\overline{M}\text{ and }\overline{M_{\Theta}}=M_{\Theta}
\qquand
\overline{HN}=\overline{H}\cdot\overline{N}\text{ and }\overline{N_{\Theta}}=N_{\Theta}
$$
for all $M \in \cM$, $N \in \cN$, and $\Theta \in \{\Theta^\pm\}$. Moreover, both of these maps are involutions.

\item[(b)] The $\mathcal{H}$-modules
$\cM$ and $\cN$ have unique $\cA$-bases $\{\underline{M}_x\}_{x \in \cF_n}$ and $\{\underline{N}_x\}_{x \in \cF_n}$ satisfying
$$\underline{M}_x=\overline{\underline{M}_x}\in M_x+\sum_{w<_F x}v^{-1}\ZZ[v^{-1}]M_w
\qquand
\underline{N}_x=\overline{\underline{N}_x}\in N_x+\sum_{w<_F x}v^{-1}\ZZ[v^{-1}]N_w$$
where both sums are over $w\in\cF_n$. We refer to these as  the \emph{canonical bases} of $\cM$ and $\cN$.

\een
\end{corollary}

For other extensions of this result, see \cite{Lu2016}.

\begin{proof}
Since $(\cF_n,\ellfpf)$ is a quasiparabolic $\tS_n$-set,
existence in part (a) follows from \cite[Theorem 4.19]{Mar1}. The other claims follow from \cite[Propositions 3.2 and 3.8]{Mar1}.
Part (b) is \cite[Theorem 3.14]{Mar1}.
\end{proof}

Define $\m_{x,y}$ and $\n_{x,y}$ for $x,y\in \cF_n$ as the polynomials in $\ZZ[v^{-1}]$ such that
$$\underline{M}_y=\sum_{x\in  \tS_n}\m_{x,y}M_x\qquand \underline{N}_y=\sum_{x\in \tS_n}\n_{x,y}N_x.$$
Let $\mu_\m(x,y)$ and $\mu_\n(x,y)$ denote the coefficients of $v^{-1}$ in $\m_{x,y}$ and $\n_{x,y}$.
Define $\tau_\m$, $\tau_\n$: $\cF_n\to \mathcal{P}(S)$ by
\[
\tau_\m(x) = \{s\in S : sxs \le_F x\}\qquand \tau_\n(x) = \{s\in S : x\le_F sxs\}
\]
where $S = \{s_1,s_2,\dots,s_n \} \subset \tS_n$.
Finally, let $\omega_\m$: $\cF_n\times\cF_n\to\ZZ$ be the map with
\[
\omega_\m(x,y)=\begin{cases}\mu_\m(x,y)+\mu_\m(y,x)&\tau_\m(x)\not\subset\tau_\m(y)\\
0&\tau_\m(x)\subset\tau_\m(y).\end{cases}
\]
Define $\omega_\n$: $\cF_n\times\cF_n\to\ZZ$ by the same formula, but with $\mu_\m$ and $\tau_\m$ replaced by $\mu_\n$ and $\tau_\n$.

\begin{corollary}[{\cite[Theorem 3.26]{Mar1}}]
Both $\Gamma_\m=(\cF_n,\omega_\m,\tau_\m)$ and $\Gamma_\n=(\cF_n,\omega_\n,\tau_\n)$ are $\tS_n$-graphs.
\end{corollary}

The strongly connected components in a $W$-graph $\Gamma$ are called \emph{cells}. The connected components with respect to doubly-directed edges are called \emph{molecules}.
It is generally an interesting problem to determine the cells and molecules in a given $W$-graph.

If $n=2$ then one can show that
$\Gamma_\m$ and $\Gamma_\n$ both decompose into just two cells given by $\cFp$ and $\cFm$.
For even integers $n>2$,  the molecules and cells for $\Gamma_\m$ and $\Gamma_\n$ are still a mystery. 
We expect that the solution to this problem may be
related to the affine RSK-correspondence defined in \cite{CPY}.

\section{Transition formulas}\label{trans-sect}

Let $n$ be a positive integer. For $\pi \in \tS_n$ and $ r \in \ZZ$ define the sets
\be\label{phi-sets-eq}
\ba
 \Phi^-_r(\pi) &= \{ \sigma \in\tS_n : \pi \lessdot \sigma = \pi t_{ir}\text{ for some integer }i < r\text{ with }i\notin r + n\ZZ\},
\\
\Phi^+_r(\pi) &= \{ \sigma\in\tS_n : \pi \lessdot \sigma= \pi t_{rj}\text{ for some integer }j > r\text{ with }j\notin r + n\ZZ\}
.
\ea
\ee
Lam and Shimozono \cite{LamShim} proved the following \emph{transition formula} for $\tF_\pi$:

\begin{theorem}[{\cite[Theorem 7]{LamShim}}]\label{trans-thm}
If $\pi \in \tS_n$ and $r \in \ZZ$ then
$\sum_{\sigma \in \Phi^-_r(\pi)} \tF_\sigma = \sum_{\sigma \in \Phi^+_r(\pi)} \tF_\sigma$.
\end{theorem}

This result is an affine generalization of
the transition formula of Lascoux and Sch\"utzenberger \cite{Lascoux} for Schubert polynomials.
Our goal is to prove an analogue of Lam and Shimozono's formula for the symmetric functions $\Ffpf_z$,
using the results in the previous section.
For the rest of this section, we assume $n$ is even.

First, we need a technical lemma.
\begin{lemma}\label{yty-lemma}
Let $y\in\cF_n$ and let $r,t\in\tS_n$ be transpositions with $y\neq tyt$. 
Suppose $i<j \not\equiv i \modu n)$ are integers such that $t=t_{ij}$.
\ben
\item[(a)] If $y(i)\equiv j\modu{n})$, then $ryr=tyt$ if and only if $r=t$.
\item[(b)] If $y(i)\not\equiv j\modu{n})$, then $ryr=tyt$ if and only if $r\in\{t,yty\}$.
\een
\end{lemma}

Before proving this lemma, we remark that it is similar to \cite[Lemma 4.8]{RV}, 
which applies only in $S_n \subsetneq \tS_n$.
\begin{lemma}[{\cite[Lemma 4.8]{RV}}]\label{1}
Let $y \in S_n$ be a fixed-point-free involution and let $r,t\in S_n$ be transpositions with $y\neq tyt=ryr$. Then $r\in\{t,yty\}$.
\end{lemma}

Define $\varphi: \tS_n\to S_n$ to be the map with $\varphi(w)(i) = r_n(w(i))$ for $w \in \tS_n$ and $i \in [n]$, 
where $r_n(x)\in[n]$ and $r_n(x)\equiv x\modu{n})$. This is a surjective group homomorphism, and it holds that $\varphi(w)=\varphi(v)$ for $v,w \in \tS_n$ if and only if $w(i)\equiv v(i)\modu{n})$ for each $i\in\ZZ$.

\begin{proof}[Proof of Lemma~\ref{yty-lemma}]
Let $i'<j'\not\equiv i' \modu n)$ be such that $r=t_{i'j'}$, and assume $y \neq tyt=ryr$.
First suppose
$y(i)\equiv j\modu{n})$, so that $y(i)=j+mn$ for some $m\in\ZZ$. 
Then $tyt(k)=y(k)$ if and only if $k\not\in\{i,j\}+n\ZZ$ and $ryr(k)=y(k)$ if and only if $k\not\in\{i',j'\}+n\ZZ$,
so we must have $\{i',j'\}+n\ZZ = \{i,j\}+n\ZZ$ since $tyt=ryr$. There are two cases, which we discuss as follows:
\begin{enumerate}
\item[(1)] Suppose $i'=i+pn$ and $j'=j+qn$ for $p,q\in\ZZ$. Then $p=q$ and $r=t$ since
\[ryr(i')=j+2qn-pn-mn= tyt(i')=j+pn-mn.\]
\item[(2)] Suppose $i'=j+pn$ and $j'=i+qn$ for $p,q\in\ZZ$. Then $p=q$ and $r=t$ since \[ryr(i')=i+2qn-pn+mn=tyt(i')=i+pn+mn.\]
\end{enumerate}
This completes the proof of part (a).

Now suppose $y(i)\not\equiv j\modu{n})$, so that we have $\varphi(ryr)=\varphi(tyt)\neq\varphi(y)$. 
By Lemma~\ref{1} we have $\varphi(r)\in\{\varphi(t),\varphi(yty)\}$. If $\varphi(r)=\varphi(t)$, then we can assume $r=t_{i,j+mn}$ for some $m\in\ZZ$, in which case 
\[ ryr(i) = t_{i,j+mn}(y(j))+mn= y(j) + mn \quand  tyt(i)= t_{ij}(y(j))= y(j)\] 
since $y(j)\notin \{i,j\} + n\ZZ$. But since $ryr=tyt$, we then must 
have $m=0$ and $r=t$. 
Instead assume
$\varphi(r)=\varphi(yty)=\varphi(t_{y(i),y(j)})$. Then $\{i',j'\} + n\ZZ = \{y(i), y(j)\} + n\ZZ$, so we have two cases: 

\begin{enumerate}
\item[(1)] Suppose $i'=y(i)+pn$ and $j'=y(j)+qn$ for $p,q\in\ZZ$. Then $p=q$ and $r=tyt$ since 
\[ryr(i')=ry(j') = r(j) + qn=j+qn=tyt(i')=tyt(y(i))+pn=j+pn.\] 
\item[(2)] Suppose $i'=y(j)+pn$ and $j'=y(i)+qn$ for $p,q\in\ZZ$. Then $p=q$ and $r=tyt$ since
\[ ryr(i')=ry(j')= r(i) + qn = i+qn=tyt(i')=tyt(y(j))+pn=i+pn.\] 
\end{enumerate}
Here $t$ fixes $y(i)$ and $y(j)$, because $\{i,j\} + n\ZZ$ and $\{y(i),y(j)\}+n\ZZ$ are disjoint. Similarly, $r$ fixes $i$ and $j$.
Thus, in either case we have $r=yty$. 
We conclude that if $tyt=ryr$ and $y(i)\not\equiv j\modu{n})$ then $r \in \{t,yty\}$.
It remains to check that if $r = yty$ and $y(i)\not\equiv j\modu{n})$ then $tyt=ryr$,
but this is straightforward.
\end{proof}

\begin{proposition}\label{covering-prop}
Let $y,z\in\cF_n$ and $\Theta \in \{ \Theta^\pm\}$ with $\sfpf(y)=\sfpf(z)=\sfpf(\Theta)$.
Fix $w\in\cAfpf(y)$ and $i<j\not\equiv i\modu{n})$ such that $w\lessdot wt_{ij}$.
Then $wt_{ij}\in\cAfpf(z)$ if and only if $y\lessdot_F z=t_{ij}yt_{ij}$.
\end{proposition}

\begin{proof}
Since $\sfpf(y)=\sfpf(z)=\sfpf(\Theta)$, we have $w^{-1}\Theta w=y$. If $wt_{ij}\in\cAfpf(z)$, then we have $z=(wt_{ij})^{-1}\Theta wt_{ij}=t_{ij}w^{-1}\Theta wt_{ij}=t_{ij}yt_{ij}$ and $\ellfpf(z)=\ell(wt_{ij})=\ell(w)+1=\ellfpf(y)+1$. Hence, $y\lessdot_F z=t_{ij}yt_{ij}$.
Conversely, if $y\lessdot_F z=t_{ij}yt_{ij}$, then $\ell(wt_{ij})=\ell(w)+1=\ellfpf(y)+1=\ellfpf(z)$ and $(wt_{ij})^{-1}\Theta wt_{ij}=t_{ij}yt_{ij}=z$, so $wt_{ij}\in\cAfpf(z)$.
\end{proof}

\begin{proposition}\label{toggle-prop}
Let $y\in\cF_n$ and $w\in\cAfpf(y)$. Fix $i<j\not\in\{i,y(i)\}+n\ZZ$ such that $w\lessdot wt_{ij}$ and $\ellfpf(t_{ij}yt_{ij})\neq\ellfpf(y)+1$. Then there exists a unique transposition $t_{kl}\neq t_{ij}$ 
(with $k<l\not \equiv k \modu n$)) such that $wt_{ij}t_{kl}\in\cAfpf(y)$ and $\ellfpf(t_{kl}yt_{kl})\neq\ellfpf(y)+1$. Actually, $t_{kl}=t_{y(j),y(i)}$ and $y(j)<y(i)$.
\end{proposition}

\begin{proof}
It is easy to check that $t_{kl}=t_{y(i),y(j)}$ has the required properties, so it suffices to prove uniqueness. But since $wt_{ij}t_{kl}\in\cAfpf(y)$, we have $t_{ij}yt_{ij}=t_{kl}yt_{kl}\neq y$, so by Lemma~\ref{yty-lemma}, we have $t_{kl}=t_{y(j),y(i)}$. Since $w\lessdot wt_{ij}$ and $\ellfpf(t_{ij}yt_{ij})\neq\ellfpf(y)+1$, we have $w(y(i))<w(y(j))$ so $y(j)<y(i)$.
\end{proof}

Recall that $R(\tS_n) = \{ ws_iw^{-1} : (w,i) \in \tS_n \times [n]\}$ is the set of all transpositions $t_{ij} \in \tS_n$.

\begin{theorem}\label{atom-bijection}
Suppose $y, z \in \cF_n$ and $y\lessdot_F z$. The map $(\pi,t) \mapsto \pi t$ is a bijection from the set of pairs $(\pi ,t) \in \cAfpf(y) \times R(\tS_n)$ with $\pi \lessdot \pi t $ and $z=tyt$ to the set of FPF-atoms  $\cAfpf(z)$.
\end{theorem}

\begin{proof}
The surjectivity follows from Corollary~\ref{atom-cover-cor} and Proposition~\ref{covering-prop}. 
To prove that the map is injective, suppose $t_{ij}, t_{kl}\in R(\tS_n)$ with $i<j$, $k<l$, and $\pi_1,\pi_2\in\cAfpf(y)$ are such that $\pi_1t_{ij}=\pi_2t_{kl}\in\cAfpf(z)$ and $z=t_{ij}yt_{ij}=t_{kl}yt_{kl}$. Then by Lemma~\ref{yty-lemma}, we have $t_{kl}\in\{t_{ij},yt_{ij}y\}$.

If $t_{kl}=yt_{ij}y$, then by the discussion in the proof of Lemma~\ref{yty-lemma}, we see that $y(i)\not\equiv j\modu{n})$ and $t_{kl}=t_{y(j),y(i)}$. 
But by Proposition~\ref{covering-prop}, we have $y\lessdot_F z$ and $\ell(t_{ij}yt_{ij})=\ell(t_{kl}yt_{kl})=\ell(y)+2$. By Corollary~\ref{length-t-cor}, we have $\ell(yt_{ij})=\ell(t_{ij}y)=\ell(y)+1=\ell(yt_{kl})=\ell(t_{kl}y)$. Therefore $y(i)<y(j)$.
But since $\pi_2=\pi_1t_{ij}t_{y(i),y(j)}\in\cAfpf(y)$, Theorem~\ref{atom-thm} implies that we must have $i<j<y(j)<y(i)$,
$i<y(j)<j<y(i)$,
$y(i)<j<y(j)<i$, or $y(i)<y(j)<j<i$, which are all contradictions.
\end{proof}
For $y \in \cF_n$ and $ r \in \ZZ$, we define
\[\ba
 \Pi^-(y,r) &= \{ z \in \cF_n : y \lessdot_F z=t_{ir}y t_{ir} \text{ for some }i < r\text{ with }i\notin \{r,y(r)\} + n\ZZ\},
\\
\Pi^+(y,r) &= \{ z \in \cF_n : y \lessdot_F z=t_{rj} y t_{rj}\text{ for some }j > r\text{ with }j\notin \{r,y(r)\} + n\ZZ\}.
\ea\]

\begin{lemma}\label{phi-subset}
If $y\in\cF_n$ and $p<q=y(p)$ then $\Pi^+(y,p)\subset\Pi^+(y,q)$ and $\Pi^-(y,q)\subset\Pi^-(y,p)$.
\end{lemma}
\begin{proof}
If $y\lessdot_F z=t_{p,i}yt_{p,i}$ where $p<i$ then by Proposition~\ref{length-t-conj-prop} and Corollary~\ref{length-t-cor}, we have $q=y(p)<y(i)$ and $z=t_{q,y(i)}yt_{q,y(i)}\in\Pi^+(y,q)$. Hence $\Pi^+(y,p)\subset\Pi^+(y,q)$. The other inclusion follows similarly.
\end{proof}


\begin{theorem} If $y \in \cF_n$ and $p,q \in \ZZ$ are such that $p< q = y(p)$ then 
\be\label{id11-eq} \sum_{z \in \Pi^-(y,p)} \Ffpf_z = \sum_{z \in \Pi^+(y,q)} \Ffpf_z.\ee
\end{theorem}

This formula is the fixed-point-free analogue of \cite[Theorem 4.15]{MZ}. 
The structure of our proof is similar but more self-contained,
as we do not depend on computer arguments, which were essential in \cite{MZ}.

\begin{proof}

Lam and Shimozono's transition formula, Theorem~\ref{trans-thm}, implies that
\[\label{id-1-eq}
\sum_{\pi \in \cAfpf(y)}\(\sum_{\substack{i < p\not\equiv i\modu{n})   \\ \pi \lessdot \pi t_{ip}}} \tF_{\pi t_{ip}}
+
\sum_{\substack{i < q\not\equiv i\modu{n})  \\ \pi \lessdot \pi t_{iq}}} \tF_{\pi t_{iq}}\)
=
\sum_{\pi \in \cAfpf(y)}
\(
\sum_{\substack{i>p\not\equiv i\modu{n})  \\ \pi \lessdot \pi t_{pi}}} \tF_{\pi t_{pi}}
+
\sum_{\substack{i>q\not\equiv i\modu{n})   \\ \pi \lessdot \pi t_{qi}}} \tF_{\pi t_{qi}}
\)
\]
Certain terms $\tF_{\pi t_{ip}}$ and $\tF_{\pi t_{iq}}$ on the left  reappear on the right as $\tF_{\pi t_{pi}}$ and $\tF_{\pi t_{qi}}$ when $i\in \{p,q\} + n\ZZ$. Canceling these terms transforms
the previous equation  to
\be\label{id1-eq}
\sum_{\pi \in \cAfpf(y)}\(\sum_{\substack{i < p   \\ \pi \lessdot \pi t_{ip}}} \tF_{\pi t_{ip}}
+
\sum_{\substack{i < q  \\ \pi \lessdot \pi t_{iq}}} \tF_{\pi t_{iq}}\)
=
\sum_{\pi \in \cAfpf(y)}
\(
\sum_{\substack{p<i  \\ \pi \lessdot \pi t_{pi}}} \tF_{\pi t_{pi}}
+
\sum_{\substack{q < i   \\ \pi \lessdot \pi t_{qi}}} \tF_{\pi t_{qi}}
\)
\ee
where the inner sums range over integers $i \notin \{p,q\} + n\ZZ$.

We discuss the left side of \eqref{id11-eq}. 
Fix $z \in \cF_n$ with $\sfpf(y) = \sfpf(z)$.
If $z \in \Pi^-(y,p)$, then it follows by Lemmas~\ref{yty-lemma} and Theorem~\ref{atom-bijection} that there exists a unique integer $i \notin \{p,q\} + n\ZZ$ with either $i<p$ and $\pi\lessdot \pi t_{ip}\in \cAfpf(z)$
or $y(i)<q$ and $\pi \lessdot \pi t_{y(i)q} \in \cAfpf(z)$. 
Conversely, suppose $\pi \in \cAfpf(y)$ and $i,j \notin \{p,q\}+n\ZZ$ are such that $i<p$ and $j<q$ and $\pi \lessdot \pi t_{ip}$ and $\pi \lessdot \pi t_{jq}$. Proposition~\ref{covering-prop} shows that
$\pi t_{ip}$ (respectively, $\pi t_{jq}$) is an FPF-atom for  $z$ if and only if
 $y\lessdot_F z =t_{ip}yt_{ip}$ (respectively, $y\lessdot_F z= t_{jq}yt_{jq}$), in which case $z \in \Pi^-(y,p)$ by \cite[Theorem 5.15]{RV} and Lemma~\ref{phi-subset}.
Thus the left side of \eqref{id11-eq} can be written as 
\[\sum_{\pi \in \cAfpf(y)}\(\sum_{\substack{i < p   \\ \pi \lessdot \pi t_{ip}}} \tF_{\pi t_{ip}}
+
\sum_{\substack{i < q  \\ \pi \lessdot \pi t_{iq}}} \tF_{\pi t_{iq}}\)
-\sum_{(\pi,i,j) \in \cN^-} \tF_{\pi t_{ij}}
\]
where $\cN^-$ is the set of triples $(\pi,i,j) \in \cAfpf(y) \times \ZZ \times \ZZ$ 
with $i < j \in \{p,q\}$ and $i \notin \{p,q\} + n \ZZ$ and
$\pi\lessdot \pi t_{ij}$ and $\ellfpf(t_{ij}yt_{ij})\neq\ellfpf(y)+1$.

Similarly, the right side of \eqref{id11-eq} can be written as 
\[
\sum_{\pi \in \cAfpf(y)}
\(
\sum_{\substack{p<i  \\ \pi \lessdot \pi t_{pi}}} \tF_{\pi t_{pi}}
+
\sum_{\substack{q < i   \\ \pi \lessdot \pi t_{qi}}} \tF_{\pi t_{qi}}\)
-\sum_{(\pi,i,j) \in \cN^+} \tF_{\pi t_{ij}}
\]
where $\cN^+$ is the set of triples $(\pi,i,j) \in \cAfpf(y) \times \ZZ \times \ZZ$ 
with $j>i \in \{p,q\}$ and $j \notin \{p,q\} + n \ZZ$ and
$\pi\lessdot \pi t_{ij}$ and $\ellfpf(t_{ij}yt_{ij})\neq\ellfpf(y)+1$.

It therefore suffices to show that $\sum_{(\pi,i,j) \in \cN^-} \tF_{\pi t_{ij}}=\sum_{(\pi,i,j) \in \cN^+} \tF_{\pi t_{ij}}$.
Let 
$\cN \supset \cN^\pm$ be the set of triples $(\pi,i,j) \in \cAfpf(y)\times \ZZ \times \ZZ$
with $i< j \not\equiv i \modu n)$ and $\pi \lessdot \pi t_{ij}$ and $\ellfpf(t_{ij}yt_{ij})\neq\ellfpf(y)+1$.
Given $(\pi,i,j) \in \cN$, let $t_{kl}=t_{y(i),y(j)}$; then by Proposition~\ref{toggle-prop}, we have 
$ \pi \neq \pi t_{ij} t_{kl} \in \cAfpf(y)$ and $\ellfpf(t_{kl}yt_{kl})\neq\ellfpf(y)+1$,
so we define
  $\theta(\pi,i,j) = (\pi t_{ij} t_{kl},k,l).$
This gives a map $\theta : \cN \to \cN$.
It is straightforward to check that $\theta$ is actually an involution,
and it follows from Proposition~\ref{toggle-prop}
that $\theta$ restricts to a bijection $\cN^- \to \cN^+$.
Since  $\theta(\pi_1,i,j) = (\pi_2,k,l)$ implies that $\pi_1 t_{ij} = \pi_2 t_{kl}$,
we get $\sum_{(\pi,i,j) \in \cN^-} \tF_{\pi t_{ij}} = \sum_{(\pi,i,j) \in \cN^+} \tF_{\pi t_{ij}}$
as needed.
\end{proof}

\begin{example}
When $n=2$, 
the theorem is equivalent to the identity $ |\Pi^-(y,p)| = | \Pi^+(y,q)|$.
\end{example}
\begin{example}
Again let $n=4$,
 $y=t_{1,6}t_{3,8}=[6,-3,8,-1]\in\cF_4$, and choose $(p,q)=(1,6)$. Then  
\[
\ba
\Pi^-(y,p)&=\{[-5,-4,9,10],[4,-5,10,1]\}=\{t_{3,9}t_{4,10},t_{1,4}t_{3,10}\},
\\
\Pi^+(y,q)&=\{[7,8,-3,-2],[8,-1,6,-3]\}=\{t_{1,7}t_{2,8},t_{1,8}t_{3,6}\},
\ea
\]
and $\Ffpf_{[-5,-4,9,10]}+\Ffpf_{[4,-5,10,1]}=\Ffpf_{[7,8,-3,-2]}+\Ffpf_{[8,-1,6,-3]}=12m_{1^5}+6m_{21^3}+3m_{2^21}+2m_{31^2}+m_{32}.$
\end{example}

In \cite{Paw2016}, Pawlowski uses Lam and Shimozono's transition formula 
to derive an effective
recursion for computing the Schur expansion of cohomology classes of certain positroid varieties.
We wonder if the affine involution transition formulas proved here and in \cite{MZ}
could be used in similar applications.


\end{document}